\documentclass[12pt]{amsart}
\usepackage{main}
\usepackage{graphicx}
\usepackage{float}
\usepackage{hyperref}
\usepackage{epstopdf}

\def\hat{\widehat}
\def\wtilde{\widetilde}

\def\bra{\langle}
\def\cet{\rangle}
\definecolor{pasgre}{rgb}{0.67, 0.87, 0.67}
\definecolor{inchworm}{rgb}{0.7, 0.93, 0.36}
\definecolor{aqua}{rgb}{0.0, 1.0, 1.0}
\def\ligt{}

\def\ligtc{}

\def\reali{\mathbb{R}}

\def\zeta{\mathbb{Z}_+}
\newtheorem*{theorem*}{Theorem}

\def \p {\partial}

\def \ba {\begin {eqnarray*} }	
\def \ea {\end {eqnarray*} }
\def \beq {\begin {eqnarray}}
\def \eeq {\end {eqnarray}}

\newcommand{\lig}[1]{\colorbox{yellow}{$\displaystyle #1$}}
\newcommand{\ligc}[1]{\colorbox{aqua}{$\displaystyle #1$}}

\def\lig{}
\def\ligc{}
\newcommand{\sijoitus}[2]%
{\operatornamewithlimits{\Bigl/}_{\!\!\!#1}^{\,#2}}

\title[Discrete regularization of wave equation]{Discrete regularization and convergence of the inverse problem for 1+1 dimensional wave equation}
\author{Jussi Korpela}
\author{Matti Lassas}
\author{Lauri Oksanen}
\date{February 13, 2018}
\begin{document}
\maketitle
\begin{abstract}
An inverse boundary value problem for the 1+1 dimensional wave equation
$(\p_t^2 - c(x)^2 \p_x^2)u(x,t)=0,\quad x\in\reali_+$ is considered. We give a discrete regularization strategy to recover wave speed $c(x)$ when we are given the boundary value of the wave,
$u(0,t)$, that is produced by a single pulse-like source. The regularization strategy gives an approximative wave speed $\wtilde c$, satisfying a Hölder type estimate 
$\| \wtilde c-c\|\leq C \epsilon^{\gamma}$, where $\epsilon$ is the noise level. 

\end{abstract}
\section{Introduction}
We consider an inverse boundary value problem for the wave equation 
$$
(\frac {\p^2}{\p t^2} - c(x)^2 \frac {\p^2}{\p x^2}) u(t,x) = 0,$$
and introduce a discrete regularization strategy to recover the sound speed $c(x)$ by using the knowledge of perturbed and discetized Neumann-to-Dirichlet map $\wtilde \Lambda_{N_1}$.
Our approach is based on the Boundary Control method \cite{Belishev1987, Belishev1992, Tataru1995}.

A variant of the Boundary Control method, called the iterative time-reversal control method, was introduced in \cite{Bingham2008}.
The method was later modified in \cite{Dahl2009}
to focus the energy of a wave at a fixed time
and in \cite{Lauri2013} 
to solve an inverse obstacle problem for a wave equation. 
In \cite{jusa1} we introduced a modification to the iterative time-reversal control method
that is tailored for the 1+1 dimensional wave equation.

The novelty in this paper is that we analyze the effect of the discretization in the regularized solution of the inverse problem. We give a direct discrete regularization method for the
non-linear inverse problem for the wave equation. The result contains an
explicit (but not necessarily optimal) convergence rate.

By referring to direct methods for non-linear problems we mean the explicit construction of non-linear map to solve the problem
without resorting to a local optimization method. In our case the map is given by (\ref{Rstategia}), shown below.
The advantage of direct approaches
is that they do not suffer from the possibility that the algorithm converges to a local minimum.
In particular, they do not require a priori knowledge
that the solution is in a small neighbourhood of a given function.

Classical abstract regularization theory is explained in \cite{Engl1996}. The iterative regularization
of both linear and non-linear inverse problems and convergence rates are discussed
in a Hilbert space setting in \cite{Bissantz2004, Hanke1997, Hohage2008, LuPeRa2007, Mathe2008} and in a Banach space setting in \cite{Hofmann2007, Kaltenbacher2006, Kaltenbacher2008, Kirsch1996, Ramlau2008, Ramlau2006, Resmerita2005}. In section \ref{pamelapulkkinen} we compare our regularization strategy to Morozov's discrepancy principle (MDP). In the context of abstract regularization theory, this principle has been discussed, e.g., in \cite{Scherzer1993}.  

There are currently only a few regularized direct methods for non-linear inverse problems.
An example is a regularisation algorithm for
the inverse problem for the conductivity equation in \cite{KLMS2009}.
 Also, a direct regularized inversion for blind
deconvolution is presented in \cite{Justen2006}.

\section{Regularization Strategy}
\subsection{Continuity of forward map}
We define 
\begin{align}
\label{normipoika}
&\norm{c}_{C^k(M)}=
\sum_{p=0}^{k}\sup_{x\in (0,\infty)}|\frac{\p^p c}{\p x^p}(x)|,
\end{align}
where $M$ denotes the half axis $M=[0,\infty)\subset \R$. We denote the set of bounded $C^k(M)$ -functions by 
\begin{align*}
C^k_b(M)=\{c\in C^k(M); \norm{c}_{C^k(M)}< \infty
\}.
\end{align*}
 Let $C_0,C_1,\lig{L_0,L_1},m>0$ and define the space of \ligt{k times differentiable} velocity functions
\begin{align}
\label{nopeudet}
 \lig{\mathcal V^k}=&\,\{c\in \lig{C^k} (M);
C_0\le c(x)\le C_1 ,\\\nonumber &\,\norm{c}_{\lig{C^k}(M)}\le m,\, c-1 \in \lig{C_0^k([L_0,L_1])} \}.
\end{align}
Here $C_0^k([L_0,L_1])$ is the subspace of functions in $C^k_b(M)$ that are supported on $[L_0,L_1]$.
Let 
\begin{align}
\label{aikapoika}
T> \frac{\lig{L_1}}{C_0}.
\end{align}
For $c\in\lig{\mathcal V^2}$ and 
$f\in L^2(0,2T)$, the boundary value problem 
\begin{align}
\label{dartwader}
&(\frac {\p^2}{\p t^2} - c(x)^2 \frac {\p^2}{\p x^2}) u(\ligc{x,t}) = 0 \quad \text{in $\ligc{M\times (0,2T})$},
\\\nonumber
&\p_x u(\ligc{0,t}) = f(t),
\\\nonumber
&u|_{t = 0} =0,\quad  \p_t u|_{t=0} = 0,
\end{align}
has a unique solution $u=u^f\in H^1(\ligc{M\times (0,2T)})$. Using this solution we define the Neumann-to-Dirichlet operator $ \Lambda= \Lambda_c$,
\begin{align}
\label{lam}
 \Lambda : L^{2}(0,2T) \to L^{2}(0,2T),\quad
 \Lambda f=u^f|_{x=0}.
\end{align}
For a Banach space $E$ we define
\begin{align*}
\mathcal{L} (E)=\{A:E \to E;
A\text{ is linear and continuous}\}.
\end{align*}
Let $\ligc{Z=C^2_b(M)}$  and $Y=\mathcal{L} (L^2(0,2T))$. The operator $\mathcal A$ is defined in the domain $\mathcal D(\mathcal A)=\mathcal V^3$ by setting
\begin{align}
\label{apina}
\mathcal A:\lig{\mathcal D(\mathcal A)}\subset \ligc Z\to\mathcal R (\mathcal{A})\subset Y,\quad 
 \mathcal A(c)= \Lambda_c. 
\end{align}
The notation in (\ref{apina}) means that the range \ligt{$\mathcal R (\mathcal{A})=\mathcal A(\mathcal V^3)$} and the domain \ligt{$\mathcal D(\mathcal A)$} are equipped with the topologies of $Y$ and $\ligc Z$, respectively. Note that the maps (\ref{lam}) and (\ref{apina}) are continuous (see \cite{jusa1}).

\subsection{Regularization strategies with discretizatized measurements}

Let $T$ be as in (\ref{aikapoika}) and $N\in \zeta$. For $n\in\{1,2,3,...,2N\}$ we define the basis functions  as
\begin{align}
\label{kantafunktio}
 \phi_{n,N}(t) = \Big(\frac{N}{T}\Big)^{\frac{1}{2}}1_{[\frac{(n-1)T}{N},\frac{nT}{N} )}(t), \quad t \in [0,2T).
\end{align}
Note that the functions ${\phi_{n,N}}$ are orthonormal in $L^2(0,2T)$. 
Having (\ref{kantafunktio}) we define the space of piecewise constant functions as 
\begin{align}
\label{pcsfjoukko}
\mathcal P^N=span\big\{\phi_{1,N},...,\phi_{2N,N}\big\}\subset L^2(0,2T) 
\end{align}
and 
an orthogonal projection as
\begin{align}
\label{pcsfjoukko2}
P^N:L^2(0,2T)\to\mathcal P^N,\quad P^N(f)=\sum_{j=1}^{2N} \bra f,\phi_{j,N}\cet
_{L^2(0,2T)}\phi_{j,N}(t).
\end{align}
Let $\Lambda$ be as in (\ref{lam}). Using (\ref{pcsfjoukko2}) we define 
\begin{align}
\label{pallokala}
\Lambda_N=P^N\Lambda P^N.
 \end{align}\\
Let $E$ be a Banach space and $\mathcal H\in E$. We denote 
\begin{align}
\label{pallo}
 \mathcal{B}_E(\mathcal H,\epsilon)=\{\wtilde {\mathcal H}\in E : \norm{\wtilde{\mathcal H}-\mathcal H}_E < \epsilon \}.
\end{align}

\subsubsection{A model for a single discrete and noisy measurement}
\label{huippumalli1}
Let $\epsilon_0>0$ and we define
\begin{align}
\label{porraspossu}
l_0(\epsilon_0)=\big\lfloor\frac{4}{7}\log_2 \epsilon_0^{-1}\big\rfloor
\end{align}
and 
\begin{align}
\label{porraspossu2}
N_0(\epsilon_0)=2^{l_0}.
\end{align}
Let $N=2^l\ge N_0$, where $l\in \zeta$. Let $P^N$ be as in (\ref{pcsfjoukko2}) and $\Lambda $ be as in (\ref{lam}). 
Let us define
\begin{align}
\label{porras}
 H(t) = 
\begin{cases}
1, & t\ge 0, \\
0, & t<0.
\end{cases}
\end{align}
Let us define 
 \begin{align}
\label{mittauspoika}
 \wtilde m_{N,\epsilon_0}=P^N\Lambda H + n_{N,\epsilon_0} ,
\end{align}
in which $n_{N,\epsilon_0} \in \mathcal P^N$ represents the error and $\norm{n_{N,\epsilon_0} }_{L^2(0,2T)}\le \epsilon_0$. We consider the quantity $(\epsilon_0,N_0,\wtilde m_{N,\epsilon_0})$ that we call a measurement. 
Let $\mathcal A$ be as in (\ref{apina}) and $H$ be as in (\ref{porras}). We define 
\begin{align}
\label{apina4}
\mathcal A_0:\mathcal D(\mathcal A_0)\subset Z\to\mathcal R (\mathcal A_0)\subset L^2(0,2T),\quad 
 \mathcal A_0(c) =  \mathcal A(c) H= \Lambda H, 
\end{align}
where $\mathcal D(\mathcal A_0)=\mathcal V^3$ . 
Our main result on the reconstruction of $c(x)$ from the measurement $(\epsilon_0,N_0,\wtilde m_{N,\epsilon_0})$ is given by the following theorem. 
 
\begin{theorem}
\label{kaiken_teoria2b}
For the operator $\mathcal A_0:\mathcal \mathcal V^3\subset Z\to L^2(0,2T)$, there exists an admissible regularization strategy  $\mathcal R^{(0)}_{N_0,\alpha_0}$ with the choice of parameter 
\begin{align*}
\alpha_0(\epsilon_0)=a_0\epsilon_0^{\frac{4}{45}},
\end{align*}
satisfying the following: For every $c\in\ligc{\mathcal V^3}$
there are $\wtilde\epsilon_0>$, $a_0>0$, and $C>0$ such that
\begin{align*} 
&\sup\Big\{\norm{\mathcal R^{(0)}_{N_0,\alpha_0}\wtilde m_{N,\epsilon_0}-c}_{\ligc Z}:
\wtilde m_{N,\epsilon_0}\in\mathcal P^{N},N=2^l\ge N_0(\epsilon_0),
\\\nonumber&\qquad  \norm{\wtilde m_{N,\epsilon_0} -P^N\Lambda H}_{L^2(0,2T)} \le \epsilon_0\Big\}\le
C\epsilon_0 ^{\ligc\gamma_0},
\end{align*}
for all $\epsilon_0 \in (0,\wtilde\epsilon_0)$. Here $\gamma_0 =\frac{1}{270}$ and $ N_0(\epsilon_0)$ is as in (\ref{porraspossu2}).
\end{theorem}
An explicit bound $\wtilde\epsilon_0$ and the value for constant $a_0$ are given in the proof.
The proof of Theorem \ref{kaiken_teoria2b} is given in Section \ref{convergense}.
\subsubsection{A model for several discrete and noisy measurements}
\label{virhesikapossu}
Let  
\begin{align}
\label{virheoperator}
\mathcal{E}_{N_1}:\mathcal P^{N_1}\to\mathcal P^{N_1},
\end{align}
where $N_1\in\zeta$.
Having $\Lambda_{N_1}$ as in (\ref{pallokala}) we define a discrete and noisy measurement operator
\begin{align}
\label{akuaku}
& \wtilde\Lambda_{N_1}: L^2(0,2T) \to L^2(0,2T), 
\quad 
\\\nonumber&\wtilde\Lambda_{N_1}f = 
\begin{cases}
\Lambda_{N_1}f+\mathcal{E}_{N_1}f, & f \in \mathcal P^{N_1}, \\
\quad\qquad 0, & f \in (\mathcal P^{N_1})^\perp.
\end{cases}
\end{align}
Note that $\mathcal P^{N_1}\cup (\mathcal P^{N_1})^\perp=L^2(0,2T)$. 
With data corresponding to several boundary measurements $(\epsilon_1, N_1, \wtilde\Lambda_{N_1})$ we get
the following results with improved error estimates.
\begin{theorem}
\label{kaiken_teoria2}
For the operator $\mathcal A:\mathcal \mathcal V^3\subset Z\to Y$, there exists an admissible regularization strategy  $\mathcal R^{(1)}_{N_1,\alpha_1}$ with the choice of parameter 
\begin{align*}
\alpha_1(\epsilon_1)=a_1\epsilon_1^{\frac{4}{9}}
\end{align*}
that satisfies the following: For every $c\in\ligc{\mathcal V^3}$
there are $\wtilde\epsilon_1>0$, $a_1>0$, and $C>0$ such that
\begin{align*} 
&\sup\Big\{\norm{\mathcal R^{(1)}_{N_1,\alpha_1}\wtilde\Lambda_{N_1}-c}_{\ligc Z}:
N_1\ge\epsilon_1^{-4},\quad \wtilde\Lambda_{N_1}\in \mathcal L(\mathcal P^{N_1}),
\\\nonumber&\qquad \norm{\wtilde\Lambda_{N_1}-\Lambda_{N_1} }_{{Y}}\le \epsilon_1\Big\}\le
C\epsilon_1 ^{\ligc\gamma_1},
\end{align*}
for all $\epsilon_1 \in (0,\wtilde\epsilon_1)$. Here $\gamma_1 =\frac{1}{54}$.
\end{theorem}
An explicit bound $\wtilde\epsilon_1$ and the value for constant $a_1$ are given in the proof.
The proof of Theorem \ref{kaiken_teoria2} is given in Section \ref{convergense}.
We will give explicit choices for $\mathcal R^{(0)}_{N_0,\alpha_0}$ in formula (\ref{Rstategia2}) and for $\mathcal R^{(1)}_{N_1,\alpha_1}$ in formula (\ref{Rstategia2b}) below. For the convenience of the reader we give a short summary of the regularization strategy. 
Assume that we are given $\wtilde \Lambda_{N}\in\mathcal L(\mathcal P^{N_1})\subset Y$, that is,
 the discrete Neumann-to-Dirichlet map for the unknown wave speed $c(x)$ with measurements errors. Then
 the regularization strategy is obtained by the following steps:

\begin{enumerate}
\item Using operator $\wtilde\Lambda_{N_1}$ in (\ref{sope3}) we constructed a source that produces a wave such that
$u^{f_{\alpha,r}}(t,x)|_{t=T}$ is close to the indicator function   $1_{\mathcal M(r)}(x)$ of the 
domain of influence ${\mathcal M(r)}$---see (\ref{siiri}).  
 \item Using sources $\wtilde f^N_{\alpha,r}$ we approximately compute the 
volumes $V(r)$ of the domains of influences---see (\ref{operaattorit4}), (\ref{sope2}), and Proposition \ref{PSN}. 
  \item Using finite differences we compute approximate values of the 
  derivatives of the volumes of the domain influences $\p_r V(r)$---see   (\ref{timojutila2c}). 
 \item We interpolate the obtained values of $\p_r V(r)$. This determines
 the approximate values of the wave speed $v(r)$ in the travel time coordinates---see (\ref{operaattorit5}) and (\ref{sebastianahoc}). 
  \item Finally, we change the coordinates from travel time coordinates to Euclidean
  coordinates to obtain the approximate values of the wave speed $c(x)$ for $x\in M$---see (\ref{kenguru}).
\end{enumerate}
\subsection{Previous literature}
From the point of view of uniqueness questions, the inverse problem
for the 1+1 dimensional wave equation is equivalent to the one-dimensional inverse boundary spectral problem. The latter problem was
thoroughly studied in the 1950s  \cite{GeLe1951, Krein1951, Marchenko1950} and we refer to \cite[pp. 65--67]{Kabanikhin2005} for a historical overview. In the
1960s Blagove{\v{s}}{\v{c}}enski{\u\i} \cite{Blagovevsvcenskiui1969, Blagovevsvcenskiui1971} developed an
approach to solving the inverse problem for the 1+1 dimensional wave
equation without reducing the problem to the inverse boundary spectral
problem. This and later dynamical methods have the advantage over
spectral methods that they only require data on a finite time
interval. Applications of one-dimensional inverse probems have been discussed widely in
\cite{Blagovevsvcenskiui1966, Kabanikhin2005,KKL2001}. 

The method in the present paper is a variant of the Boundary Control
 method  that was pioneered by M. Belishev \cite{Belishev1987}
and developed by  M. Belishev and Y. Kurylev \cite{BeKu1987,Belishev1992} in the late 1980s and early 1990s.
Of crucial importance for the method is the result by D.\ Tataru
\cite{Tataru1995} concerning
a Holmgren-type uniqueness theorem for non-analytic coefficients.
The Boundary Control method for multidimensional inverse problems has been summarized in \cite{Belishev1997, KKL2001},
and considered for 1+1 dimensional scalar problems in \cite{BeKa1989, BeRFi1994, jusa1} and for  multidimensional scalar problems
in \cite{Katchalov1998, KKLM2004, Kury1995, LaOk2014, LaOk2014b}.
For systems it has been considered in \cite{Kurylev2009, Kurylev2006,Kurylev2015}.
Stability results for the method have been considered in \cite{Anderson2004}
and \cite{Katsuda2007}, and
computational implementations in
\cite{Belishev1999,Belishev2016,Hoop2017,Ivanov2016,Pestov2010}.
An application of the method to blockage detection in water pipes is in preparation \cite{eemeli}.

The inverse problem for the wave equation can also be solved by using complex geometrical optics solutions. These solutions were developed in the context of elliptic
inverse boundary value problems \cite{Sylvester1987}, and in \cite{Nachman1988} they were employed to solve an inverse
boundary spectral problem. Local stability results can be proven using
(real) geometrical optics solutions \cite{Bellassoued2011, StUh1998, uulmanni2013}, and in \cite{lauri2012} a stability result was
proved by using ideas from the Boundary Control method together with complex
geometrical optics solutions. 

 There is an important method
based on Carleman estimates \cite{Bukhgeuim1981}, often called the Bukhgeim-Klibanov method after its founders, that can be used to show stability
results requiring only a
single instance of boundary values, when the initial data for the wave equation is non-vanishing. 
We mention the interesting recent computational work \cite{Baudouin} that is based on this method, and also
another reconstruction method that uses a single measurement \cite{beilina2012approximate, Beilina2008}. 
This method is based
on a reduction to a non-linear integro-differential equation, and
there are several papers on
how to solve this equation (or an approximate version of it)---see \cite{Klibanov2017,  Klibanov2015} for recent results
including computational implementations. 
 
\subsection{Notations}
We will define $\mathcal R^{(1)}_{N_1,\alpha_1}$ as a discrete version of the regulation strategy given in \cite{jusa1}. For that we recall some notation from  \cite{jusa1}.  

We denote the indicator function of a set $E$ by 
$$
1_E(x) = 
\begin{cases}
1, & x \in E,
\\
0, & \text{otherwise}.
\end{cases}
$$
We define 
\begin{align}
\label{operaattorit}
J : L^2(0,2T) \to L^2(0,2T), \quad
J f(t) = \frac{1}{2} \int_0^{2 T} 1_\blacktriangle (t,s)f(s) ds,&
\end{align}
where
$
\blacktriangle = \{(t,s) \in (0,2T)^2;\ \text{$t + s \le 2 T$ and $s > t > 0$}\}.
$
We define the time reversal operator as
\begin{align}
\label{operaattorit2}
R : L^2(0,2T) \to L^2(0,2T), \quad
R f(t) = f(2 T - t),
\end{align}
and the projections as
\begin{align}
\label{operaattorit3}
 P_r f(t) = 1_{(T-r,T)}(t)f(t), \quad r \in [0,T].
\end{align}
Using (\ref{operaattorit}), (\ref{operaattorit2}), and (\ref{operaattorit3}) we define that 
\begin{align}
\label{Koo}
K : Y \to Y, \quad
KL = JL-RLRJ
\end{align}
and
\begin{align}
\label{Hoo}
\boldsymbol{H} &:Y \mapsto C([0,T],Y), 
\quad\boldsymbol{H}L(r) =     
P_r(K L)P_r.
\end{align}

We define a regularized inversion with cutoff as
\begin{align}
\label{Zalfa}
{Z}_{\alpha} : Y \to Y,
\quad
Z_\alpha(L) = \eta_Y(L,\alpha) (L+\alpha)^{-1},
\end{align}
where $\eta_Y : Y \times (0,\infty) \to \R$ is, for example, a continuous function that satisfies
$\eta_Y(L,\alpha) = 1$ when $d(L,Y^+) \le \alpha / 4$ and 
$\eta_Y(L,\alpha) = 0$ when $d(L,Y^+) \ge \alpha / 2$. Here 
$$
d(L,Y^+) = \inf \{\norm{L-L^+}_Y;\ \text{$L^+ \in Y$ is positive semidefinite}\}.
$$
We denote by $\boldsymbol{Z}_{\alpha}$ the lift of ${Z}_{\alpha}$
to $C([0,T],Y)$, that is,
$\boldsymbol{Z}_{\alpha}(L)(r)$ is ${Z}_{\alpha}(L(r))$.
Moreover, we define $b(t) = 1_{(0,T)}(t) (T-t)$ and
\begin{align}
\label{innerproduct}
& \boldsymbol{S}: C([0,T],Y) \to C([0,T]),
\quad \boldsymbol SL(r) =\bra L(r)P_r b,b\cet _{L^2(0,2T)}.
\end{align}

We define the travel time coordinates by
\begin{align}
\label{def_tau}
\tau:[0,\infty)\to[0,\infty), \quad \tau(x) = \int_0^x \frac{1}{c(t)} dt, \quad x \in M,
\end{align}
and the domain of influence 
\begin{align}
\label{siiri}
\mathcal M(r) = \{x \in M; \tau(x) \le r\}, r \ge 0.
\end{align}
The function $\tau$ 
is strictly increasing and we denote its inverse by $\chi$.
Moreover, $V(r)$ denotes the volume of $\mathcal M(r)$ with respect to the measure $dV = c^{-2}dx$, 
where $c$ is the speed of sound in (\ref{dartwader}).
From \cite[Eq. (21)]{jusa1} we see that 
\begin{align}
\label{operaattorit4}
V = \lim_{\alpha \to 0} 
(\boldsymbol S \circ \boldsymbol Z_\alpha \circ \boldsymbol H)(\Lambda).
\end{align}
Moreover, according to \cite[Eq. (19), (20)]{jusa1}, the speed of sound in travel time coordinates $v = c \circ \chi$ satisfies
\begin{align}
\label{operaattorit5}
v(r) = \frac 1 {\p_r V(r)}, \quad \chi(r) = \int_0^r v(t) dt, \quad r > 0. 
\end{align}
Thus $c$ can be computed from $V$. We will next recall how the formula (\ref{operaattorit5}) is regularized in \cite{jusa1}. For small $h>0$ we consider the partition
\begin{align} 
\label{jako}
(0,T)=(0,h) \cup[h,2h)\cup[2h,3h)\cup...\cup
[N_hh-h,N_hh)\cup [N_hh,T),
\end{align}
where $N_h\in \mathbb{N}$ satisfies $T-h\le N_hh < T$.
We define a discretized and regularized approximation of the derivative operator $\p_r$ by
\begin{align} 
\label{operaattori6}
D_{h}:C([0,T])\to L^\infty(0,T), 
\end{align}
\begin{displaymath}
D_{h}(f)(t)= \left\{ \begin{array}{cl}
\frac{f(h)}{h}, & \textrm{if\quad $t \in (0,h)$},\\
\frac{f(jh+h
)-f(jh)}{h}, & \textrm{if\quad $t \in [jh,jh+h)$},\\
 \frac{f(T)-f(N_hh)}{T-N_hh},& \textrm{if\quad $t \in [N_hh,T)$.}\\ 
\end{array} \right.
\end{displaymath}
Let us have 
\begin{align}
\label{arskasika}
\vartheta (t;a,b) = a 1_{(-\infty, a)}(t) + t 1_{[a, b]}(t)+
b 1_{( b,\infty)}(t).
\end{align}
We define an inversion with a cutoff that takes into account the a priori bounds in (\ref{nopeudet}) by 
\begin{align}
\label{rampo}
z : \R \to \R
\quad
z(t) &= \frac{1}{\vartheta (t;C_0,C_1)}.
\end{align}
We denote by $\boldsymbol z$ the lift of $z$
to $L^\infty(0,T)$, that is,
$\boldsymbol z(f)(t) = z(f(t))$.
We define the extension by one
\begin{align}
\label{arska}
E : L^\infty(0,T) \to L^\infty(0,\infty), 
\quad Ef(t) = 
\begin{cases}
f(t), & t \in (0,T), \\
1, & \text{otherwise},
\end{cases}
\end{align}
and set $W = E \circ \boldsymbol z$. We define 
\begin{align}
\label{arskapiks}
\tilde \chi : L^\infty(0,\infty) \to C(0,\infty), \quad
\tilde \chi(f)(r) = \int_0^r f(t) dt.
\end{align}
Note that having $f>0$ in (\ref{arskapiks}) we have it that $\tilde \chi(f)$ is a strictly increasing function. 
Having $L_0,L_1$, as in (\ref{nopeudet}), we define 
\begin{align} 
\label{operaattori16}
\theta_\R: L^\infty(0,\infty) \to L^\infty(\reali), 
\end{align}
\begin{displaymath}
\theta_\R(f)(t)= \left\{ \begin{array}{cl}
1, & \textrm{if\quad $t \in (-\infty,L_0]$},\\
f(t), & \textrm{if\quad $t \in (L_0,L_1)$},\\
 1,& \textrm{if\quad $t \in [L_1,\infty)$.}\\ 
\end{array} \right.
\end{displaymath}
Having $f>0$ and using (\ref{arskapiks}) and (\ref{operaattori16}) we define
\begin{align}
\label{hiiri} 
\Phi : L^\infty(0,\infty) \to L^\infty(\reali),
\quad
\Phi(f) = \theta_\R(f \circ (\tilde \chi(f))^{-1}).
\end{align}
Let us define $\eta\in C^\infty (\reali)$ by
\begin{align} 
\label{konvooluutti}
\ligc{ \eta(x)= \left\{ \begin{array}{cl}
  C\exp\big (\frac{1}{x^2-1}\big ), & \textrm{if\quad $x \in (-1,1)$},\\
0, & \textrm{if\quad $|x|\ge 1$},\\
\end{array} \right.}
\end{align}
\ligtc{where the constant $C>0$ is selected so that $\int_\reali \eta (x)=1$.} For $\nu >0$ we define
\begin{align} 
\label{molliolli}
\ligc{  \eta_\nu (x)=\frac{1}{\nu }\eta \Big (\frac{x}{\nu } \Big ).}
\end{align}
By using convolution we define a smooth approximation to a given function $f\in L^\infty(\reali )$ by setting
\begin{align} 
\label{smootapp}
\Gamma_\nu:L^\infty(\reali)\to C^\infty (\reali ),\quad \Gamma_\nu ( f)= \eta_\nu \ast f.
\end{align}
 Using (\ref{Hoo}),
(\ref{Zalfa}), (\ref{innerproduct}), (\ref{operaattori6}), (\ref{arska}), (\ref{hiiri}), and (\ref{smootapp}) we define the family of operators for the regularization strategy used in \cite{jusa1} by
\begin{align}
\label{Rstategia}
&\mathcal R_{\alpha}:Y\to Z,
\\\nonumber&\mathcal R_{\alpha}=\Gamma_\nu\circ\Phi\circ W\circ D_{h}\circ\boldsymbol{S}\circ \boldsymbol{Z}_{\alpha}\circ\boldsymbol{H},
\end{align}
where $\nu=C\epsilon^{\frac{1}{54}}$, $h=C\epsilon^{\frac{1}{18}}$ and $\alpha(\epsilon)=2^{\frac{13}{9}}T^{\frac{4}{9}}\epsilon^{\frac{4}{9}}$. Note that in \cite{jusa1} we considered perturbations of the Neumann-to-Dirichlet operator of the form  
\begin{align}
\label{virhemitta}
&\wtilde\Lambda=\Lambda+\mathcal{E}, 
\end{align}
where $\mathcal{E}\in Y$
models the measurement error and $\norm{\mathcal{E}}_Y \le \epsilon$. 
Below we will introduce a discretized version of regularization strategy (\ref{Rstategia}) that takes in discretized measurements. To this end we start with auxiliary lemmas.
\subsection{The proofs of the main results}
\label{convergense}

\begin{lemma}
\label{DLa} 
Let $N\in \zeta$. Let $\Lambda $ be as in (\ref{lam}) and $\Lambda_N$ as in (\ref{pallokala}). Then we have
\begin{align*}
 \norm{\Lambda_{N}-\Lambda }_{{Y}}\le CN^{-\frac{1}{4}}.
\end{align*}
\end{lemma}
Here $C=C(T)>0$ depends on $T$.
\begin{proof}
By (\ref{lam}) and the trace theorem we have
\begin{align}
\label{Lamestim}
 \norm{\Lambda}_{L^2(0,2T)\to H^{\frac{1}{2}}(0,2T)}
\le C_{Lam}.
\end{align}
By (\ref{pcsfjoukko2}) we have
\begin{align}
\label{virheprojektiolle7}
 \norm{I-P^N}_{L^2(0,2T)\to L^2(0,2T)}
\le2.
\end{align}
Let $f\in H^2(0,2T)$. By (\ref{pcsfjoukko2}) and $H^2(0,2T)\hookrightarrow C^1([0,2T])$ we have
\begin{align}
\label{virheprojektiolle3}
 \norm{f -P^N f}_{L^\infty (0,2T)}\le
 \frac{T}{N}\norm{f}_{C^1([0,2T])}.
\end{align}
Thus 
\begin{align}
\label{virheprojektiolle4}
 \norm{f -P^N f}_{L^2(0,2T)}\le(2T)^{\frac{1}{2}}\norm{f -P^N f}_{L^\infty (0,2T)}
\le(2T)^{\frac{1}{2}} \frac{T}{N}\norm{f}_{C^1([0,2T])}.
\end{align}
By (\ref{virheprojektiolle4}) and having $C_{sob}=\norm{I}_{H^2(0,2T)\to C^1([0,2T])}$ we get
\begin{align}
\label{virheprojektiolle5}
 \norm{I-P^N}_{H^2(0,2T)\to L^2(0,2T)}
\le(2T)^{\frac{1}{2}} \frac{T}{N}C_{sob}.
\end{align}
Using interpolation theory with (\ref{virheprojektiolle7}) and (\ref{virheprojektiolle5}) we have
\begin{align}
\label{virheprojektiolle8}
 \norm{I-P^N}_{H^{\frac{1}{2}}(0,2T)\to L^2(0,2T)}
\le 2^{\frac{3}{4}}\big((2T)^{\frac{1}{2}} \frac{T}{N}C_{sob}\big)^{\frac{1}{4}}.
\end{align}
For self-adjoint operators $P^N$ and $\Lambda$ we have
\begin{align}
\label{virheprojektiolle10}
 \norm{\Lambda -\Lambda P^N }_{{Y}}=
  \norm{(\Lambda -\Lambda P^N )^*}_{Y}=
  \norm{(I - P^N)\Lambda }_{{Y}}.
\end{align}
By (\ref{Lamestim}) and (\ref{virheprojektiolle8}) we have
\begin{align}
\label{virheprojektiolle12}
 \norm{\Lambda -\Lambda P^N }_{{Y}}=\norm{(I - P^N)\Lambda }_{{Y}}\le
 2^{\frac{3}{4}}\big((2T)^{\frac{1}{2}} \frac{T}{N}C_{sob}\big)^{\frac{1}{4}}C_{Lam}.
\end{align}
Thus
\begin{align}
\label{virheprojektiolle13}
 \norm{P^N \Lambda -P^N \Lambda P^N }_{{Y}}\le
 2^{\frac{3}{4}}\big((2T)^{\frac{1}{2}} \frac{T}{N}C_{sob}\big)^{\frac{1}{4}}C_{Lam}.
\end{align}
By (\ref{virheprojektiolle10}), (\ref{virheprojektiolle12}), and (\ref{virheprojektiolle13}) 
we have
\begin{align}
\label{virheprojektiolle14}
 \norm{\Lambda -P^N \Lambda P^N }_{{Y}}\le
 2^{\frac{7}{4}}\big((2T)^{\frac{1}{2}} \frac{T}{N}C_{sob}\big)^{\frac{1}{4}}C_{Lam}\le
CN^{-\frac{1}{4}}.
\end{align}
\end{proof}
\begin{proposition}
\label{DLb} 
Let $\epsilon>0$, $N\in \zeta$ and $N\ge\epsilon^{-4}$ . Let $\Lambda $ be as in (\ref{lam}) and $\Lambda_N$ be as in (\ref{pallokala}). Assume that $\wtilde\Lambda_{N}\in \mathcal{B}_Y (\Lambda_N,\epsilon)$, then 
\begin{align*}
 \norm{\wtilde\Lambda_{N}-\Lambda }_{{Y}}\le C_2\epsilon.
\end{align*}
\end{proposition}
Here $C_2=C_2(T)>0$ depends on $T$.
\begin{proof}

Using Lemma \ref{DLa} and having $N\ge\epsilon^{-4}$
we get
\begin{align}
\label{virheprojektiolle15}
 \norm{\wtilde\Lambda_{N}-\Lambda }_{{Y}}\le  \epsilon +
 CN^{-\frac{1}{4}}\le (C+1)\epsilon.
\end{align}
\end{proof}
 Let $J$ be as in (\ref{operaattorit})  and using (\ref{pcsfjoukko2}) we define
\begin{align}
\label{pridekulkue}
J_N=P^NJP^N.
\end{align}
\begin{lemma}
\label{JN}
Let $N\in\zeta$. Let $J$ be as in (\ref{operaattorit}) and $J^N$ be as in (\ref{pridekulkue}). Then we have
\begin{align*}
 \norm{J-J_N}_{Y}\le
CN^{-\frac{1}{2}}.
\end{align*}
\end{lemma}
Here $C=C(T)>0$ depends on $T$.
\begin{proof}
By (\ref{operaattorit}) we have
\begin{align}
\label{JnormiH}
 \norm{J}_{L^2(0,2T)\to H^{1}(0,2T)}\le C_{J}.
\end{align}
Using interpolation with (\ref{virheprojektiolle7}) and (\ref{virheprojektiolle5}) we have
\begin{align}
\label{virheprojektiolle8b}
 \norm{I-P^N}_{H^{1}(0,2T)\to L^2(0,2T)}
\le 2^{\frac{1}{2}}\big((2T)^{\frac{1}{2}} \frac{T}{N}C_{sob}\big)^{\frac{1}{2}}.
\end{align}
By (\ref{JnormiH}) and (\ref{virheprojektiolle8b}) we have
\begin{align}
\label{virheprojektiolle9b}
 \norm{(I-P^N)J}_{L^{2}(0,2T)\to L^2(0,2T)}
\le C_{J} 2^{\frac{1}{2}}\big((2T)^{\frac{1}{2}} \frac{T}{N}C_{sob}\big)^{\frac{1}{2}}.
\end{align}
We have
\begin{align}
\label{virheprojektiolle9c}
 \norm{J(I-P^N)}_{L^{2}(0,2T)\to L^2(0,2T)}=
 \norm{(I-P^N)J^*}_{L^{2}(0,2T)\to L^2(0,2T)}.
\end{align}
By (\ref{JnormiH}), (\ref{virheprojektiolle9b}), and (\ref{virheprojektiolle9c})
we get
\begin{align}
\label{virheprojektiolle9d}
 \norm{J-P^NJP^N}_{Y}\le
 \norm{(I-P^N)J}_{Y}+\norm{P^N}_{Y}\norm{J(I-P^N)}_{Y}
\le CN^{-\frac{1}{2}}.
\end{align}
\end{proof}

Let $J$ be as in (\ref{operaattorit}), $J^N$ be as in (\ref{pridekulkue}), and $R$ be as in (\ref{operaattorit2}). We define 
\begin{align}
\label{Koo2}
K^N : Y \to Y, \quad
K^N L = J_NL-RLRJ_N.
\end{align}
By (\ref{operaattorit3}) and (\ref{Koo2}) we define   
\begin{align}
\label{Hn}
\boldsymbol{H}^N &:Y \to C([0,T],Y), 
\quad\boldsymbol{H}^NL(r) =     
P_r(K^N L)P_r.
\end{align}
\begin{proposition}
\label{HN}
Let $\epsilon\in (0,1)$, $\kappa_3>0$, and $N\ge\kappa_3\epsilon^{-4}$. Let $\Lambda $ be as in (\ref{lam}), $\boldsymbol{H}$ be as in (\ref{Hoo}), and $\boldsymbol{H}^N$ be as in (\ref{Hn}). Assume that $ \wtilde\Lambda_N\in \mathcal{B}_Y ( \Lambda,\epsilon)$, then we have  
\begin{align*}
 \norm{\boldsymbol{H}^N\wtilde\Lambda_N -\boldsymbol{H}\Lambda }_{C([0,T],Y)}\le C_3\epsilon.
\end{align*}
\end{proposition}
Here $C_3=C_3(T,\kappa_3)>0$ depends on $T$ and $\kappa_3$.
\begin{proof}
By (\ref{Koo}) and (\ref{Koo2}) we get
\begin{align}
\label{Kvirhe10b}
K^N\wtilde\Lambda_N -K\wtilde\Lambda_N=
R\wtilde\Lambda_NR(J-J_N)+(J_N-J)\wtilde\Lambda_N.
\end{align}
Using (\ref{operaattorit2}), we have $\norm{R}_Y\le 1$. Using \cite[Theorem 5]{jusa1}, we have $\norm{\Lambda}_Y \le M_1<\infty$. By (\ref{Kvirhe10b}) and Lemma \ref{JN} we have
\begin{align}
\label{Kvirhe10}
 \norm{\boldsymbol{H}^N\wtilde\Lambda_N(r) -\boldsymbol{H}\wtilde\Lambda_N (r) }_Y\le\norm{K^N\wtilde\Lambda_N -K\wtilde\Lambda_N}_{Y}\le
2(M_1+1)CN^{-\frac{1}{2}}.
\end{align}
By \cite[Proposition 1]{jusa1} we have
\begin{align}
\label{Kvirhepapru}
 \norm{\boldsymbol{H}\wtilde\Lambda_N -\boldsymbol{H}\Lambda}_{C([0,T],Y)}\le
T\epsilon.
\end{align}
By (\ref{Kvirhe10}) and (\ref{Kvirhepapru}), when $\epsilon\in (0,1)$ and $N\ge\kappa_3\epsilon^{-4}$ we have 
\begin{align}
\label{tohtorisaksala}
 \norm{\boldsymbol{H}^N\wtilde\Lambda_N -\boldsymbol{H}\Lambda }_{C([0,T],Y)}\le 2(M_1+1)C\Big(\frac{1}{\kappa_2}\Big)^{\frac{1}{2}}\epsilon^{2}
 +T\epsilon\le C_3\epsilon.
\end{align}
\end{proof}

Let $P_r$ be as in (\ref{operaattorit3}), $b$ be as in (\ref{innerproduct}), and $P^N$ be as in (\ref{pcsfjoukko2}). We define 
\begin{align}
\label{innerproduct2}
& \boldsymbol S^N: C([0,T],Y) \to C([0,T]),
\quad \boldsymbol S^NL(r) =\bra P^N L(r)P_r b,b\cet _{L^2(0,2T)}.
\end{align}
\begin{lemma}
\label{SN}
Let $L\in C([0,T],Y)$ and $\boldsymbol{S}$ be as defined in (\ref{innerproduct}). Then
\begin{align*}
\norm{\boldsymbol S^NL-\boldsymbol SL}_{C([0,T])}
\le \frac{T^3}{6N}\norm{L}_{C([0,T],Y)}.
\end{align*} 
\end{lemma}
\begin{proof} By using (\ref{innerproduct}), (\ref{innerproduct2}) with the self-adjointness of operator $P^N$ we get 
\begin{align*}
|\boldsymbol S^NL(r)-\boldsymbol SL(r)|\le \norm{L(r)}_Y  \norm{P_rb}_{L^2(0,2T)}\norm{P^Nb-b}_{L^2(0,2T)}.
\end{align*}
We have $\norm{P_rb}_{L^2(0,2T)}\le(\frac{T^3}{3})^{\frac{1}{2}}$ and $\norm{P^Nb-b}_{L^2(0,2T)}\le(\frac{T^3}{12N^2})^{\frac{1}{2}}.$
\end{proof}
Let $K$ be as in (\ref{Koo}) and $K^N$ be as in (\ref{Koo2}). Let $\boldsymbol{H}$ be as in (\ref{Hoo}) and $\boldsymbol{H}^N$ be as in (\ref{Hn}). 
Let $\boldsymbol{Z}_{\alpha}$ be as in (\ref{Zalfa}). For $r\in[0,T]$ and $\alpha>0$ we denote
\begin{align}
\label{elvisonkunkku}
& Z_{\alpha,r}=\boldsymbol{Z}_{\alpha}(\boldsymbol{H} \Lambda)(r),
 \\\nonumber &\wtilde Z_{\alpha,r}=\boldsymbol{Z}_{\alpha}(\boldsymbol{H}^N\wtilde\Lambda_N)(r).
\end{align} 
Let $\boldsymbol{S}$ be as in (\ref{innerproduct}). Using (\ref{elvisonkunkku}) we denote
\begin{align}
\label{sope2}
& s_\alpha = \boldsymbol S \circ \boldsymbol Z_\alpha \circ \boldsymbol H\Lambda,
 \\\nonumber &\wtilde s_\alpha = \boldsymbol S \circ \boldsymbol Z_\alpha \circ \boldsymbol{H}^N\wtilde\Lambda_N.
\end{align} 
Let $\boldsymbol S^N$ be as in (\ref{innerproduct2}). Using  (\ref{elvisonkunkku}) we denote
\begin{align}
\label{sope3}
\wtilde s^N_\alpha = \boldsymbol S^N \circ \boldsymbol Z_\alpha \circ \boldsymbol{H}^N\wtilde\Lambda_N,\quad \wtilde s^N_\alpha (r)
 =\bra P^N\wtilde Z_{\alpha,r} P_r b,b\cet _{L^2(0,2T)}.
\end{align} 
\begin{proposition}
\label{PSN}
Let $\epsilon\in (0,1)$ and $\kappa_4>0$. Let $\alpha=2\epsilon^{\frac{4}{9}}$ and $N\ge\kappa_4\epsilon^{-4}$. Let $s_\alpha$ be as defined in (\ref{sope2}) and $\wtilde s^N_\alpha$ be as defined in (\ref{sope3}). Let $\boldsymbol{H} \Lambda(r)\in Y$ be bounded and positive semidefinite for all $r\in[0,T]$. Assume that $\boldsymbol{H}^N\wtilde\Lambda_N\in \mathcal{B}_{C([0,T],Y)} (\boldsymbol{H} \Lambda,\epsilon)$, then we have
\begin{align*}
\norm{\wtilde s^N_\alpha-s_\alpha}_{C([0,T])}
\le C_4\epsilon^{\frac{1}{9}}.
\end{align*} 
Here $C_4=C_4(T,\kappa_3)>0$ depends on $T$ and $\kappa_3$.
\end{proposition}

\begin{proof}
We have 
\begin{align}
\label{soinintimppa}
\norm{\wtilde s^N_\alpha-s_\alpha}_{C([0,T])}
\le\norm{\wtilde s^N_\alpha-\wtilde s_\alpha}_{C([0,T])}+\norm{\wtilde s_\alpha-s_\alpha}_{C([0,T])} .
\end{align} 
Let $\boldsymbol{S}$ be as defined in (\ref{innerproduct}). We have $\norm{\boldsymbol{S}}_{C([0,T],Y)\to C([0,T])}\le \frac{T^3}{3}$, see \cite[Proposition 2]{jusa1}. Having  $\boldsymbol{H}^N\wtilde\Lambda_N\in \mathcal{B}_{C([0,T],Y)} (\boldsymbol{H} \Lambda,\epsilon)$ we get $\norm{\boldsymbol{Z}_{\alpha}(\boldsymbol{H}^N\wtilde\Lambda_N) -\boldsymbol{Z}_{\alpha}(\boldsymbol{H} \Lambda) }_{C([0,T],Y)}\le\frac{1}{2}\epsilon^{\frac{1}{9}}$, see \cite[Proposition 2]{jusa1}. Using (\ref{sope2}), for the second part of the sum in the right-hand side we get
\begin{align}
\label{rape}
\norm{\wtilde s_\alpha-s_\alpha}_{C([0,T])}
\le\norm{\boldsymbol{S}}_\Omega\norm{\boldsymbol{Z}_{\alpha}(\boldsymbol{H}^N\wtilde\Lambda_N) -\boldsymbol{Z}_{\alpha}(\boldsymbol{H} \Lambda) }_{C([0,T],Y)}\le\frac{T^3}{3}\frac{1}{2}\epsilon^{\frac{1}{9}} ,
\end{align} 
where we denote $\Omega=C([0,T],Y)\to C([0,T])$.
Using (\ref{sope2}) and (\ref{sope3}) with Lemma \ref{SN}, for the first part of the sum in the right-hand side we get
\begin{align}
\label{rape2}
\norm{\wtilde s^N_\alpha-\wtilde s_\alpha}_{C([0,T])}
\le\frac{T^3}{6N}\norm{\boldsymbol{Z}_{\alpha}(\boldsymbol{H}^N\wtilde\Lambda_N)}_{C([0,T],Y)}.
\end{align} 
We have 
\begin{align}
\label{rape3}
\norm{\boldsymbol{Z}_{\alpha}(\boldsymbol{H}^N\wtilde\Lambda_N)}_\Omega
\le\norm{\boldsymbol{Z}_{\alpha}(\boldsymbol{H}^N\wtilde\Lambda_N)-\boldsymbol{Z}_{\alpha}(\boldsymbol{H} \Lambda) }_\Omega+\norm{\boldsymbol{Z}_{\alpha}(\boldsymbol{H} \Lambda) }_\Omega,
\end{align} 
where we denote $\Omega=C([0,T],Y)$. Using \cite[Eq. (32)]{jusa1} we have
\begin{align}
\label{rape4}
\norm{\boldsymbol{Z}_{\alpha}(\boldsymbol{H} \Lambda (r)) }_Y\le\alpha^{-1}.
\end{align} 
Having $\epsilon\in (0,1)$, $N\ge\kappa_4\epsilon^{-4}$ and $\alpha=2\epsilon^{\frac{4}{9}}$ with use of (\ref{rape2}), (\ref{rape3}), and (\ref{rape4}) we get
\begin{align}
\label{rape5}
\norm{\wtilde s^N_\alpha-\wtilde s_\alpha}_{C([0,T])}
\le\frac{T^3}{6}\frac{\epsilon^4}{\kappa_4}\Big(\frac{1}{2}\epsilon^{\frac{1}{9}}+\frac{1}{2}\epsilon^{-\frac{4}{9}}\Big)\le \frac{T^3}{3\kappa_4}\epsilon^{\frac{32}{9}}.
\end{align} 
By using (\ref{soinintimppa}) (\ref{rape}) and (\ref{rape5}) we get the estimate.
\end{proof}
\begin{proof}[Proof of Theorem \ref{kaiken_teoria2}]
\label{hauki}
\hfill \break
Let us consider the measurement ($\epsilon_1, N_1, \wtilde\Lambda_{N_1}$).
There is $N_1\ge\epsilon_1^{-4}$ and $\wtilde\Lambda_{N_1}$, for which $\norm{\wtilde\Lambda_{N_1}-\Lambda_{N_1}}_{Y}\le \epsilon_1$, and Proposition  \ref{DLb}  gives us

\begin{align*}
 \norm{\wtilde\Lambda_{N_1}-\Lambda }_{{Y}}\le C_2\epsilon_1.
\end{align*}

  Let $\epsilon_1\in (0,C_2^{-1})$ and $\epsilon=C_2\epsilon_1$. We choose $\kappa_3=C_2^4$ and thus $ N_1\ge\epsilon_1^{-4}=C_2^4\epsilon^{-4}$. Having $\wtilde\Lambda_{N_1}\in \mathcal{B}_{Y} (\Lambda,\epsilon)$, Proposition \ref{HN} gives us 
\begin{align}
\label{patelaine2c}
 \norm{\boldsymbol{H}^{N_1}\wtilde\Lambda_{N_1} -\boldsymbol{H}\Lambda}_{C([0,T],Y)}\le C_3C_2\epsilon_1.
\end{align}
Let  $\epsilon_1\in(0,C_2^{-1}C_3^{-1})$ and $\epsilon=C_2C_3\epsilon_1$.  We choose $\kappa_4=C_2^4C_3^4$ and thus $ N_1\ge\epsilon_1^{-4}=C_2^4C_3^4\epsilon^{-4}$. Let $\alpha_1=2C_2^{\frac{4}{9}}C_3^{\frac{4}{9}}\epsilon_1^{\frac{4}{9}}$. Having (\ref{patelaine2c}), Proposition \ref{PSN} gives us
\begin{align}
\label{timojutilac}
\norm{\wtilde s^{N_1}_\alpha-s_{\alpha}}_{C([0,T])} \le
C_4C_3^{\frac{1}{9}}C_2^{\frac{1}{9}}\epsilon_1^{\frac{1}{9}}.
\end{align}
Let $V$ be as defined as in (\ref{operaattorit4}) and $D_{h}$ be as defined as in (\ref{operaattori6}). Let $\epsilon_1\in(0,\hat\epsilon)$, where $\hat\epsilon$ as defined in  \cite[Proposition 4]{jusa1}. Let $\kappa_5 >0$ and $h=\kappa_5\epsilon_1^{\frac{1}{18}}$. Having (\ref{timojutilac}) and using \cite[Proposition 4]{jusa1} we get 
\begin{align} 
\label{timojutila2c}
\norm{D_{h}(\wtilde s^{N_1}_\alpha)-\p_r V}_{L^\infty (0,T)}
\le C_5 C_4^{\frac{1}{2}} C_3^{\frac{1}{18}}C_2^{\frac{1}{18}}\epsilon_1^{\frac{1}{18}},
\end{align} 
where $C_5=C_5(\kappa_5)$. Note that we use the parameter $\kappa_5$ to control the size of discretization in (\ref{operaattori6}).
Let $v$ be as defined in (\ref{operaattorit5}) and $W$ be as defined in (\ref{arska}). Let us denote $\wtilde{w}^{N_1}_{\alpha}=W(D_{h}(\wtilde s^{N_1}_\alpha))$. 
Having (\ref{timojutila2c}) and using \cite[Proposition 5]{jusa1} we get 
\begin{align}
\label{sebastianahoc}
\norm{ \wtilde{w}^{N_1}_{\alpha}-v}_{L^\infty (M)} 
\le C_6 C_5 C_4^{\frac{1}{2}} C_3^{\frac{1}{18}} C_2^{\frac{1}{18}}\epsilon_1^{\frac{1}{18}}.
\end{align}
Let $\nu =\kappa_6\epsilon_1^\frac{1}{54}$ and $c\in\mathcal V^3$. Let $\Phi$ be as in (\ref{hiiri}) and $\eta_\nu$ be as in (\ref{molliolli}). Let us denote $\wtilde c^{N_1}_{\alpha}=\eta_\nu \ast\Phi(\wtilde w^{N_1}_\alpha)$. Having (\ref{sebastianahoc}) and using  \cite[Proposition 6]{jusa1} we get 
\begin{align}
\label{kenguru}
\norm{\wtilde c^{N_1}_{\alpha} -c}_{C^2(M)} 
\le C_7C_6^{\frac{1}{3}} C_5^{\frac{1}{3}} C_4^{\frac{1}{6}} C_3^{\frac{1}{54}} C_2^{\frac{1}{54}}\epsilon_1^{\frac{1}{54}},
\end{align}
where $C_7=C_7(\kappa_6)$. Note that we use the parameter $\kappa_6$ to control the support of $\eta_\nu$ in (\ref{molliolli}).
We define
\begin{align}
\label{epsilooninolla}
 \wtilde\epsilon_1=min\Big\{C_2^{-1},C_2^{-1}C_3^{-1},\hat\epsilon\Big\}.
\end{align} 

By using (\ref{Zalfa}), (\ref{operaattori6}), (\ref{arska}), (\ref{hiiri}), (\ref{smootapp}), (\ref{Hn}), and (\ref{innerproduct2}) we define 
\begin{align}
\label{Rstategia2}
& \mathcal R^{(1)}_{N_1,\alpha_1}:Y\to \ligc Z,
\\\nonumber&\mathcal \mathcal R^{(1)}_{N_1,\alpha_1}=\Gamma_\nu\circ\Phi\circ W\circ D_{h}\circ\boldsymbol{S}^{N_1}\circ \boldsymbol{Z}_{\alpha_1}\circ\boldsymbol{H}^{N_1},
\end{align}
and have the estimate
\begin{align*}
\norm{\mathcal R^{(1)}_{N_1,\alpha_1}(\wtilde\Lambda_{N_1})-c}_{\ligc Z} \le
a_1\epsilon_1^{\frac{1}{54}},
\end{align*}
when $\epsilon_1 \in (0,\wtilde\epsilon_1)$ and $a_1=C_7C_6^{\frac{1}{3}} C_5^{\frac{1}{3}} C_4^{\frac{1}{6}} C_3^{\frac{1}{54}} C_2^{\frac{1}{54}}$. 
\end{proof}
\begin{proof}[Proof of Theorem \ref{kaiken_teoria2b}]
\label{ahven}
\hfill \break
Let ${\phi_{j,N_0}}$ be as in (\ref{kantafunktio}), where $j\in\{1,2,3,...,2N_0\}$. 
Let $r\in [0,2T)$ and we define
\begin{align}
\label{translaatio}
\mathcal T_r: L^2(0,2T) \to L^2(0,2T), 
\quad  T_rf(t) = 
\begin{cases}
\,\quad 0, & t \in (0,r), \\
f(t-r), & t\in [r,2T).
\end{cases}
\end{align}

Using (\ref{mittauspoika}) and (\ref{translaatio}) we define
\begin{align}
\label{mittausaikapoika}
 \wtilde\Lambda_{\mathcal P}\phi_{j,N_0}=\Big(\frac{N_0}{T}\Big)^{\frac{1}{2}}\Big(\mathcal T_{\frac{(j-1)T}{N_0}} -\mathcal T_{\frac{jT}{N_0}}\Big)P^{N_0} \wtilde m_{N,\epsilon_0}.
\end{align}
As $\big\{\phi_{1,N_0},...,\phi_{2N_0,N_0}\big\}$ span $\mathcal P^{N_0}$ this defines a linear map
\begin{align}
\label{mittahulg}
 \wtilde\Lambda_{\mathcal P}:\mathcal P^{N_0}\to \mathcal P^{N_0}.
\end{align}
Using (\ref{mittausaikapoika}) and (\ref{mittahulg}) we define a perturbed and  discretizatized Neumann-to-Dirichlet operator

\begin{align}
\label{mittamulgi}
&\wtilde\Lambda_{N_0}:L^2(0,2T)\to L^2(0,2T),
\\\nonumber& \wtilde\Lambda_{N_0} |_{\mathcal P^{N_0}}= \wtilde\Lambda_{\mathcal P},
\\\nonumber&\wtilde\Lambda_{N_0} |_{(\mathcal P^{N_0})^\perp}=0,
\end{align}
where $\mathcal P^{N_0}\oplus (\mathcal P^{N_0})^\perp = L^2(0,2T)$. Using (\ref{mittauspoika}), (\ref{mittausaikapoika}), (\ref{mittahulg}), and (\ref{mittamulgi}) we define
\begin{align}
\label{erkkipoika}
E_{N_0} : \mathcal P^N \to Y,\quad
 E_{N_0} \wtilde m_{N,\epsilon_0}=\wtilde\Lambda_{N_0}.
\end{align}
Let $\epsilon_0>0$. Let $l_0(\epsilon_0)$ be as in (\ref{porraspossu}) and $N_0(\epsilon_0)$ be as in (\ref{porraspossu2}). Let $N=2^l\ge N_0$, where $l\in\zeta$. Let $\Lambda_{N_0} $ be as in (\ref{pallokala}) and $\wtilde\Lambda_{N_0}$ be as in (\ref{mittamulgi}). Let $P^N$ be as in (\ref{pcsfjoukko2}) and $\phi_{1,N}$ be as in (\ref{kantafunktio}). Let $\wtilde m_{N,\epsilon_0}$ be as in (\ref{mittauspoika}). Assume that $\wtilde m_{N,\epsilon_0}\in \mathcal{B}_{L^2(0,2T)} (P^N\Lambda H,\epsilon_0)$ and $\norm{n_{N,\epsilon_0} }_{L^2(0,2T)}\le\epsilon_0$. Then 
\begin{align}
\label{kumiukko}
 \norm{\wtilde\Lambda_{N_0}-\Lambda_{N_0} }_{{Y}}\le C_{P_1}\epsilon_0^{\frac{1}{5}}.
\end{align}

Here $C_{P_1}=C_{P_1}(T)>0$ depends on $T$.
Having (\ref{kantafunktio}), (\ref{porras}), and (\ref{translaatio}) we get 
 \begin{align}
\label{mittauspoikanolla}
\Lambda_{N_0}\phi_{1,{N_0}}=\Big(\frac{N_0}{T}\Big)^{\frac{1}{2}}P^{N_0}\Lambda P^{N_0}(I-\mathcal T_{\frac{T}{N_0}})H.
\end{align}
As $N=2^l\ge N_0=2^{l_0}$ we have $P^{N_0}P^N=P^{N_0}$. Using $\mathcal T_r \Lambda=\Lambda\mathcal T_r$ and $P^{N_0}\mathcal T_{\frac{T}{N_0}}=\mathcal T_{\frac{T}{N_0}}P^{N_0}$ with (\ref{mittauspoika}), (\ref{mittausaikapoika}), and (\ref{mittauspoikanolla}) we get 
 \begin{align}
\label{mittauspoikan}
 \wtilde\Lambda_{N_0}\phi_{1,{N_0}}-\Lambda_{N_0}\phi_{1,{N_0}}=\Big(\frac{N_0}{T}\Big)^{\frac{1}{2}}(I-\mathcal T_{\frac{T}{N_0}})n_{N,\epsilon_0}.
\end{align}
Using $N_0\le\epsilon_0^{-\frac{4}{5}}$, $\norm{\mathcal T_{\frac{T}{N_0}} }_{L^2(0,2T)}\le 1$ and $\norm{n_{N,\epsilon_0} }_{L^2(0,2T)}\le\epsilon_0$ we get 
\begin{align}
\label{sorsa}
 \norm{\wtilde\Lambda_{N_0}\phi_{1,{N_0}}-\Lambda_{N_0}\phi_{1,{N_0}} }_{L^2(0,2T)}\le 2T^{-\frac{1}{2}}\epsilon_0^{\frac{3}{5}}.
\end{align}
Let $f\in L^2(0,2T)$. By (\ref{pcsfjoukko2}) we get
\begin{align}
\label{virheporrasfunktiolle2}
  \norm{\wtilde\Lambda_{N_0} f-\Lambda_{N_0}  f}_{L^2(0,2T)}\le \sum_{j=1}^{2N_0} \bra f,\phi_{j,N_0}\cet
_{L^2(0,2T)}\norm{P^{N_0}\wtilde\Lambda\phi_{j,{N_0}}-P^N\Lambda\phi_{j,{N_0}}}_{L^2(0,2T)}.
\end{align}
Using (\ref{kantafunktio}), (\ref{mittausaikapoika}), (\ref{mittauspoikanolla}), and (\ref{mittauspoikan}) we have
\begin{align}
\label{virheporrasfunktiolle2cc}
 \norm{\wtilde\Lambda_{N_0} f-\Lambda_{N_0}  f}_{L^2(0,2T)}\le\norm{P^{N_0}\wtilde\Lambda\phi_{1,{N_0}}-P^{N_0}\Lambda\phi_{1,{N_0}}}_{L^2(0,2T)} \sum_{j=1}^{2{N_0}} \bra f,\phi_{j,{N_0}}\cet
_{L^2(0,2T)}.
\end{align}
Thus (\ref{sorsa}) and (\ref{virheporrasfunktiolle2cc}) gives us
\begin{align}
\label{virheporrasfunktiolle2ccg}
\norm{\wtilde\Lambda_{N_0} f-\Lambda_{N_0}  f}_{L^2(0,2T)}\le (2N_0)^{\frac{1}{2}}\norm{f}_{L^2(0,2T)}2T^{-\frac{1}{2}}\epsilon_0^{\frac{3}{5}}.
\end{align}
We have $N_0\le\epsilon_0^{-\frac{4}{5}}\le 2N_0$ and this proves (\ref{kumiukko}).

We define 
\begin{align}
\label{anttiaffauros}
 \epsilon_1=C_{P_1}\epsilon_0^{\frac{1}{5}}.
\end{align}
Using $N_0\le\epsilon_0^{-\frac{4}{5}}\le 2N_0$ and (\ref{anttiaffauros}) we get
\begin{align}
\label{anttiaffauros2}
 N_0\ge 2^{-1}C_{P_1}^{4}\epsilon_1^{-4}.
\end{align}
By (\ref{anttiaffauros2}), $\wtilde\Lambda_{N_0}\in \mathcal{B}_Y (\Lambda_{N_0},\epsilon_1)$, and Theorem \ref{kaiken_teoria2} we get 
\begin{align}
\label{katiska}
\norm{\wtilde c^{N_0}_{\alpha} -c}_{C^2(M)} 
\le C_7C_6^{\frac{1}{3}} C_5^{\frac{1}{3}} \wtilde C_4^{\frac{1}{6}} \wtilde C_3^{\frac{1}{54}} \wtilde C_2^{\frac{1}{54}}\epsilon_1^{\frac{1}{54}}.
\end{align}
Note that in the proof of Theorem \ref{kaiken_teoria2} we have assumed $N_1\ge\epsilon_1^{-4}$. After replacing this by (\ref{anttiaffauros2}), the proof is identical, only the constants $\wtilde C_2, \wtilde C_3, \wtilde C_4$ change---see (\ref{virheprojektiolle15}), (\ref{tohtorisaksala}), (\ref{rape4}), and (\ref{kenguru}). Using (\ref{anttiaffauros}) we get
\begin{align*}
\norm{\wtilde c^{N_0}_{\alpha} -c}_{C^2(M)} 
\le C_7C_6^{\frac{1}{3}} C_5^{\frac{1}{3}} \wtilde C_4^{\frac{1}{6}} \wtilde C_3^{\frac{1}{54}}\wtilde C_2^{\frac{1}{54}}C_{P_1}^{\frac{1}{54}}\epsilon_0^{\frac{1}{270}}.
\end{align*}

We define
\begin{align}
\label{epsilooninolla}
 \wtilde\epsilon_0=min\Big\{C_{P_1}^{4}\wtilde C_2^{-5},C_{P_1}^{4}\wtilde C_2^{-5}C_3^{-5},\hat\epsilon\Big\},
\end{align} 
where $\hat\epsilon$ can be specified by using \cite[Proposition 4]{jusa1}. 
By using (\ref{Zalfa}), (\ref{operaattori6}), (\ref{arska}), (\ref{hiiri}), (\ref{smootapp}), (\ref{erkkipoika}), (\ref{Hn}), and (\ref{innerproduct2}), we define 
\begin{align}
\label{Rstategia2b}
&\mathcal R^{(0)}_{N_0,\alpha_0}:\mathcal P^{N}\to \ligc Z,
\\\nonumber&\mathcal R^{(0)}_{N_0,\alpha_0}=\Gamma_\nu\circ\Phi\circ W\circ D_{h}\circ\boldsymbol{S}^{N_0}\circ \boldsymbol{Z}_{\alpha}\circ\boldsymbol{H}^{N_0}\circ E^{N_0},
\end{align}
and have the estimate
\begin{align*}
\norm{\mathcal R^{(0)}_{N_0,\alpha_0}(\wtilde m_{N,\epsilon_0})-c}_{\ligc Z} \le
a_0\epsilon_0^{\frac{1}{270}},
\end{align*}
when $\epsilon_0 \in (0,\wtilde\epsilon_0)$ and $a_0=C_7C_6^{\frac{1}{3}} C_5^{\frac{1}{3}} \wtilde C_4^{\frac{1}{6}} \wtilde C_3^{\frac{1}{54}}\wtilde C_2^{\frac{1}{54}}C_{P_1}^{\frac{1}{54}}$. 
\end{proof}

\section{Numerical examples}
\label{sec_computations}

In this section we describe a computational implementation of the regularization strategy in Theorem \ref{kaiken_teoria2}. We will also compare this with a heuristic variant of MDP---see (\ref{morozov}).
We will begin by describing how the data---that is, the noisy discretized Neumann-to-Dirichlet map $\widetilde \Lambda_{N_1}$---is simulated.

\subsection{The simulation of measurement data}

We choose $T = 0.6$ in all the simulations.
We use $k$-Wave \cite{K-Wave} to solve the
boundary value problem (\ref{dartwader}) with 
$f =\phi_{1,N_1} \in \mathcal P_{N_1}$, 
where $N_1 = 2^{10}$, and denote the solution by $u^{(sim)}$.
Recall that 
\begin{align}
\label{ekakantafunktio}
 \phi_{1,N_1}(t) = h^{-\frac{1}{2}}1_{[0,h)}(t), \quad t \in [0,2T),
\end{align} 
where $h = T / N_1$.
In order to simulate $u^{(sim)}$ for $2T$ time units, a fine discretization needs to be used, and we choose a regular mesh with $2^{13}$ spatial and $N_2 = 2^{15}$ temporal cells. 
Then we define the simulated Neumann-to-Dirichlet map, acting on the first basis function,
\begin{align}
\label{possuli}
\Lambda^{(sim)} \phi_{1,N_1}(t) = \sum_{j=1}^{2N_1} u^{(sim)}(t_j, 0) \phi_{j,N_2}(t), \quad t \in [0,2T),
\end{align}
where $t_j=\frac{(j-1)2T}{N_2}$, $j\in\{1,2,...,N_2\}$, are the temporal grid points.
The output of k-Wave is, of course, only an approximation of $u^{(sim)}$ but we do not analyze this simulation error and use the same notation for both $u^{(sim)}$ and its approximation. 

Our primary object of interest is the following discretized version of the Neumann-to-Dirichlet map 
\begin{align}
\label{lamdad}
 \Lambda_{N_1}^{(d)} : \mathcal P^{N_1} \to \mathcal P^{N_1},\quad\Lambda_{N_1}^{(d)}f=\sum_{j=1}^{2N_1}\sum_{k=1}^{j}f_k\Lambda_{j-k+1}\phi_{j,N_1}, 
\end{align}
where $\Lambda_j=\bra \Lambda^{(sim)} \phi_{1,N_1},\phi_{j,N_1}\cet_{L^2(0,2T)}$ and $f_k$, $k=1,2,\dots,2N_1$,
are the coefficients of $f$ on the basis of $\mathcal P^{N_1}$.
Observe that $\Lambda_{N_1}^{(d)} \phi_{1,N_1}$ is 
simply the projection of $\Lambda^{(sim)} \phi_{1,N_1}$ on $\mathcal P^{N_1}$, and that $\Lambda_{N_1}^{(d)}f$ is then defined by using the fact that the wave equation (\ref{dartwader}) is invariant with respect to translations in time.

We will now describe how the noise is simulated. 
Consider 
\begin{equation}
\label{the_noise}
n=(n_1,n_2,...,n_{2N_1})\in\reali^{2N_1},
\end{equation}
where
$n_j\in\mathcal{N}(0,1)$, that is, $n_j$ is a normally distributed random variable with zero mean and unit variance.
We compute a realization of $n$ by using the randn function of MATLAB,
and use the same notation for $n$ and its realization.
Let $\epsilon_0^{(d)}>0$ and define,
analogously to (\ref{lamdad}),
a noisy, discretized version of the Neumann-to-Dirichlet map 
\begin{align}
\label{lamdadvir}
 \wtilde\Lambda_{N_1}^{(d)} : \mathcal P^{N_1} \to \mathcal P^{N_1},\qquad\wtilde\Lambda_{N_1}^{(d)} f=\sum_{j=1}^{2N_1}\sum_{k=1}^{j}f_k \wtilde\Lambda_{j-k+1}\phi_{j,N_1},
\end{align}
where $\wtilde\Lambda_j=\Lambda_j+\hat n_j$ and 
\begin{align}
\label{mittavirhe_mod}
\hat n_j= \frac{\epsilon_0^{(d)}}{\norm{n}_{l^2}}n_j, 
\quad j=1,2,\dots,2N_1.
\end{align} 
Following the formulation of Theorem \ref{kaiken_teoria2},
rather than using $\epsilon_0^{(d)}>0$, we prefer to 
parametrize the noise level in terms of 
\begin{align}
\label{laskennalvirhe}
\epsilon^{(d)}_1:= \norm{\wtilde \Lambda_{N_1}^{(d)}-\Lambda_{N_1}^{(d)} }_{\mathcal P^{N_1}}.
\end{align}
In what follows, we will consider 
the quantity $(\epsilon^{(d)}_1, \wtilde \Lambda_{N_1}^{(d)})$, 
a simulated analogue of the noisy measurements in Section \ref{virhesikapossu}.

\subsection{Implementation of the regularization strategy}

For the a priori bounds in (\ref{nopeudet}) we use values $C_0=0.01$, $C_1=10$, $L_0=0.01$, and $L_1=0.6$.
The crux of the regularization strategy $\mathcal R^{(1)}_{N_1,\alpha_1}$ is the computation of the inverse in (\ref{Zalfa}). 
When starting from the simulated measurement $(\epsilon^{(d)}_1, \wtilde \Lambda_{N_1}^{(d)})$, the analogue of (\ref{Zalfa}) is to solve $X_j$ in the equation
\begin{align}
\label{yhtalot}
(P_{r_j}K^{N_1} \wtilde \Lambda_{N_1}^{(d)} P_{r_j}+\alpha_1)X_j=P_{r_j}b.
\end{align}
Here $P_r$ is the projection in (\ref{operaattorit3}), and
we choose $r_j=jh$, $j = 1,2,\dots,N_1$.
The choice of the regularization parameter $\alpha_1$ is discussed in detail below.  

We use the restarted generalized minimal residual (GMRES) method to solve the system of linear equations (\ref{yhtalot}) and choose six as the maximum number of outer iterations and 10 as the number of inner iterations (restarts). We use the initial guess $f = 0$ and the tolerance of the method is set to 1e-12.

After this we simply follow the regularization strategy 
(\ref{Rstategia2}), that is, we get an approximation of $c$ 
by setting
\begin{align}
\label{the_rec}
\wtilde c^{N_1}_{\alpha_{1}}=\Gamma_\nu(\Phi(W(D_h(\wtilde s^N_{\alpha_1})))),
\end{align}
where $\wtilde s^{N_1}_{\alpha_1} (r_j)
 =\bra X_j,b\cet _{L^2(0,2T)}$, $j = 1,2,\dots,N_1$.
The scaling of $\nu$ is chosen as follows 
\begin{align*}
\nu=0.01(\epsilon_1^{(d)})^{\frac{1}{54}}.
\end{align*}
In the numerical computations the parameter $h$ was fixed to be $h=\frac{T}{N_1}$, that is, the discrete derivative $D_h$ was computed in the grid that is used in (\ref{ekakantafunktio}) to represent the basis functions $\phi_{1,N_1}$. Observe that this deviates from the theoretical choice $h=C\epsilon_1^{\frac{1}{18}}$ used in (\ref{Rstategia}). We will describe next how the regularization parameter $\alpha_1$ is chosen, and then we will study how the error $\wtilde c^{N_1}_{\alpha_{1}} - c$ behaves as function of $\epsilon^{(d)}_1$.

\subsection{Calibration of the regularization strategy}

Recall that in Theorem \ref{kaiken_teoria2} 
the choice of regularization parameter is of the form
$\alpha_1=C_{reg1}(\epsilon_1^{(d)})^p$, where $p={\frac{4}{9}}$.
In particular, the choice is explicit apart from the constant $C_{reg1}$. In this section we choose $C_{reg1}$ so that it gives a good reconstruction of a particular velocity function $c_c$---see Figure \ref{testprofile}. Then the same constant is used in all the subsequent computational examples. 

\begin{figure}
\centering
\includegraphics[scale=0.18]
{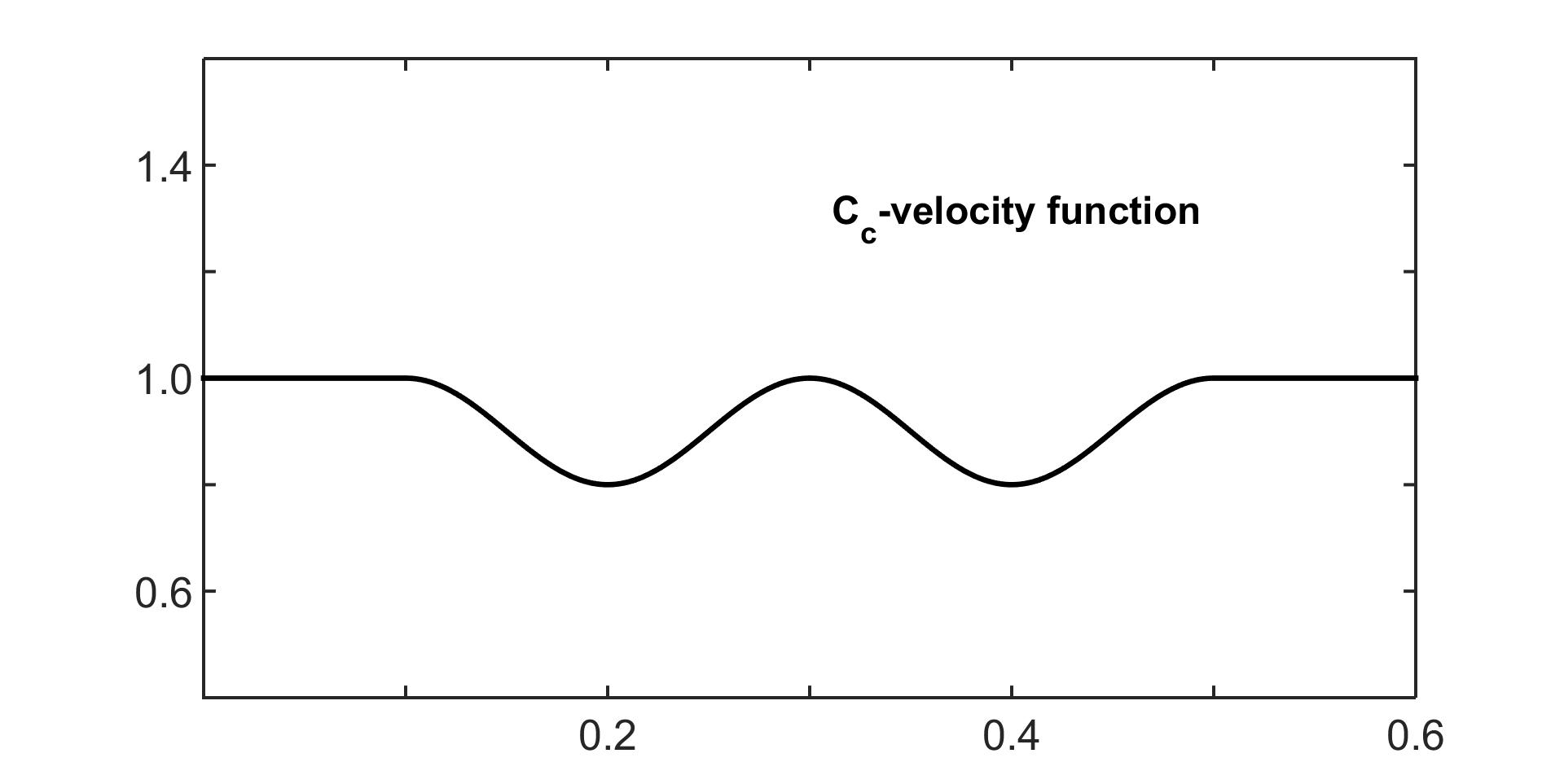}
\caption{The velocity function $c_c$
used in the calibration of the regularization strategy.
}   
\label{testprofile}
\end{figure}

In the regularization strategy we consider 10 values for measurement errors, as defined in (\ref{laskennalvirhe})
\begin{align}
\label{epsihommeli}
\epsilon^{(d)}_{1,k}\in\{k\cdot 10^{-2}|k=1,2,3,...,10\},
\end{align} 
and nine values for the multiplicative constant $C_{reg1}=10^{-j}$, $j=1,3,...,9$.
Then we consider the error in the reconstruction as a function of $j$,
\begin{align}
error(j)=\norm{\wtilde c^{N_1}_{\alpha_{j,k}} -c_c}_{L^2 (M)},
\end{align}
where for each error level, the reconstruction $\wtilde c^{N_1}_{\alpha_{j,k}}$ is computed by (\ref{the_rec}).
These computations are summarized in Figure \ref{choosing_C2}.
We see that the choice $j=4$, that is, 
\begin{align}
\label{the_alpha}
\alpha_1 = 10^{-4} (\epsilon_1^{(d)})^p, \quad p={\frac{4}{9}},
\end{align}
gives a good reconstruction on all the error levels. 
In what follows we will systematically use this choice. 

\begin{figure}
\centering
\includegraphics[scale=0.18]
{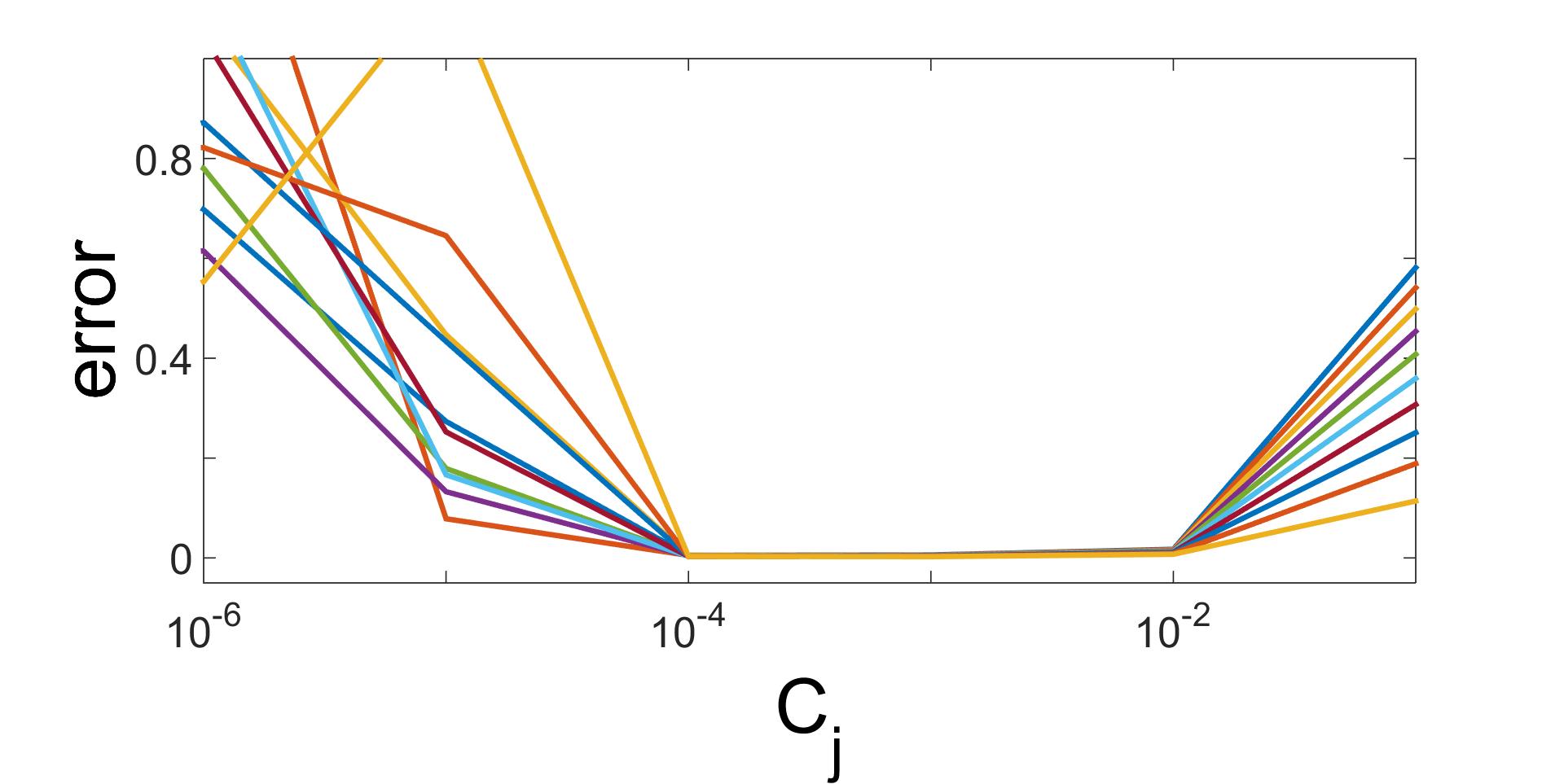}
\caption{
The reconstruction error as a function of the multiplicative constant $C_j = C_{reg1}$. Each curve corresponds to a noise level in (\ref{epsihommeli}).
As expected, the reconstruction error is monotonous as a function of the noise level: The highest line corresponds to $\epsilon_1^{(d)}=0.1$
and the lowest one to $\epsilon_1^{(d)}=0.01$. 
We also observe that the reconstruction error becomes more sensitive to the choice of $C_{reg1}$ as the noise level grows. 
}   
\label{choosing_C2}
\end{figure}

\subsection{Reconstruction results based on the analysis}
We will now consider the reconstruction (\ref{the_rec}), 
with the choice of regularization parameter (\ref{the_alpha}),
in two test cases. 
We begin with with a smooth velocity function $c_s$(see Figure \ref{estimatkunkku}), where reconstructions of two different noise levels are shown. 

 \begin{figure}
\centering
\includegraphics[scale=0.14]
{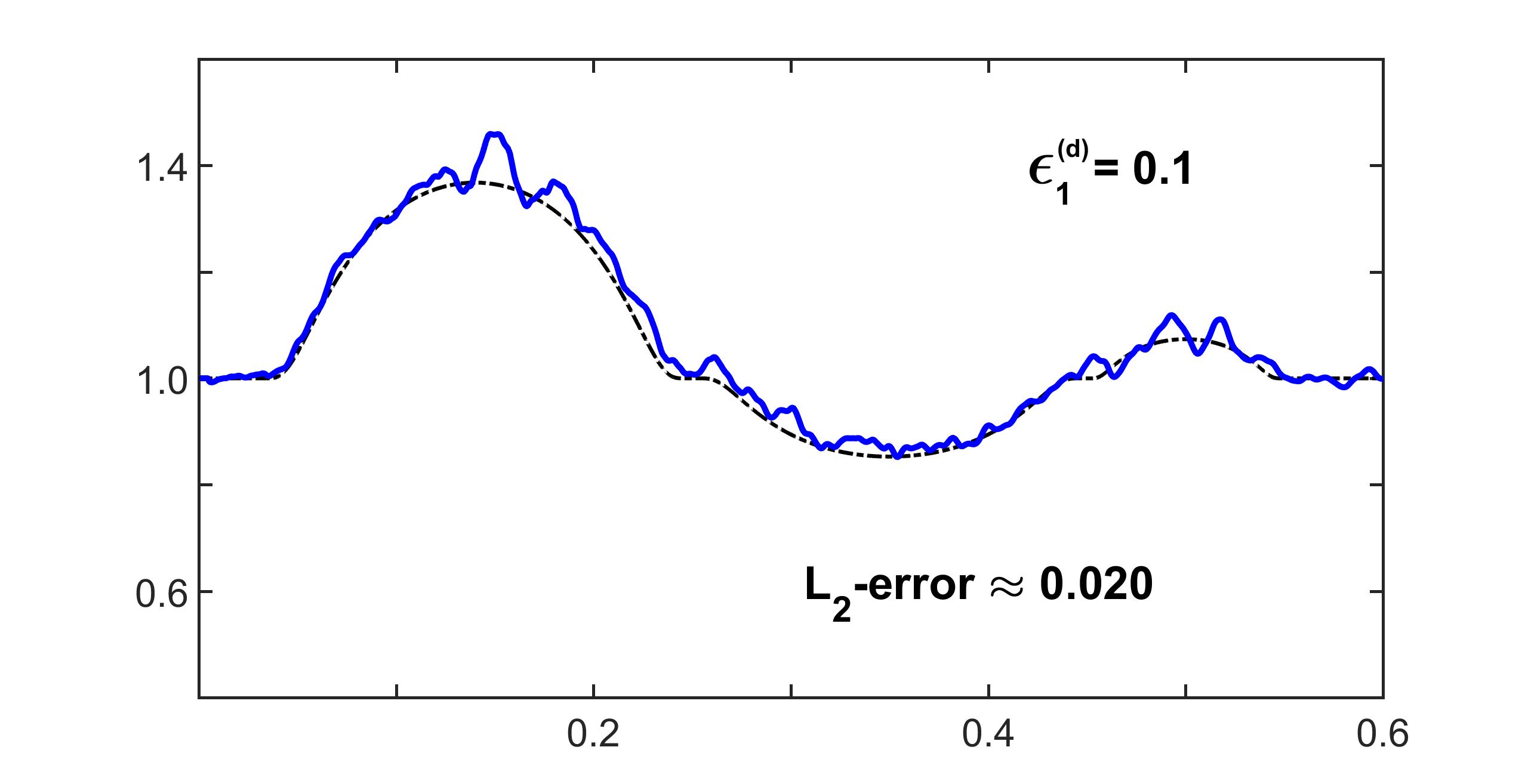}
\includegraphics[scale=0.14]
{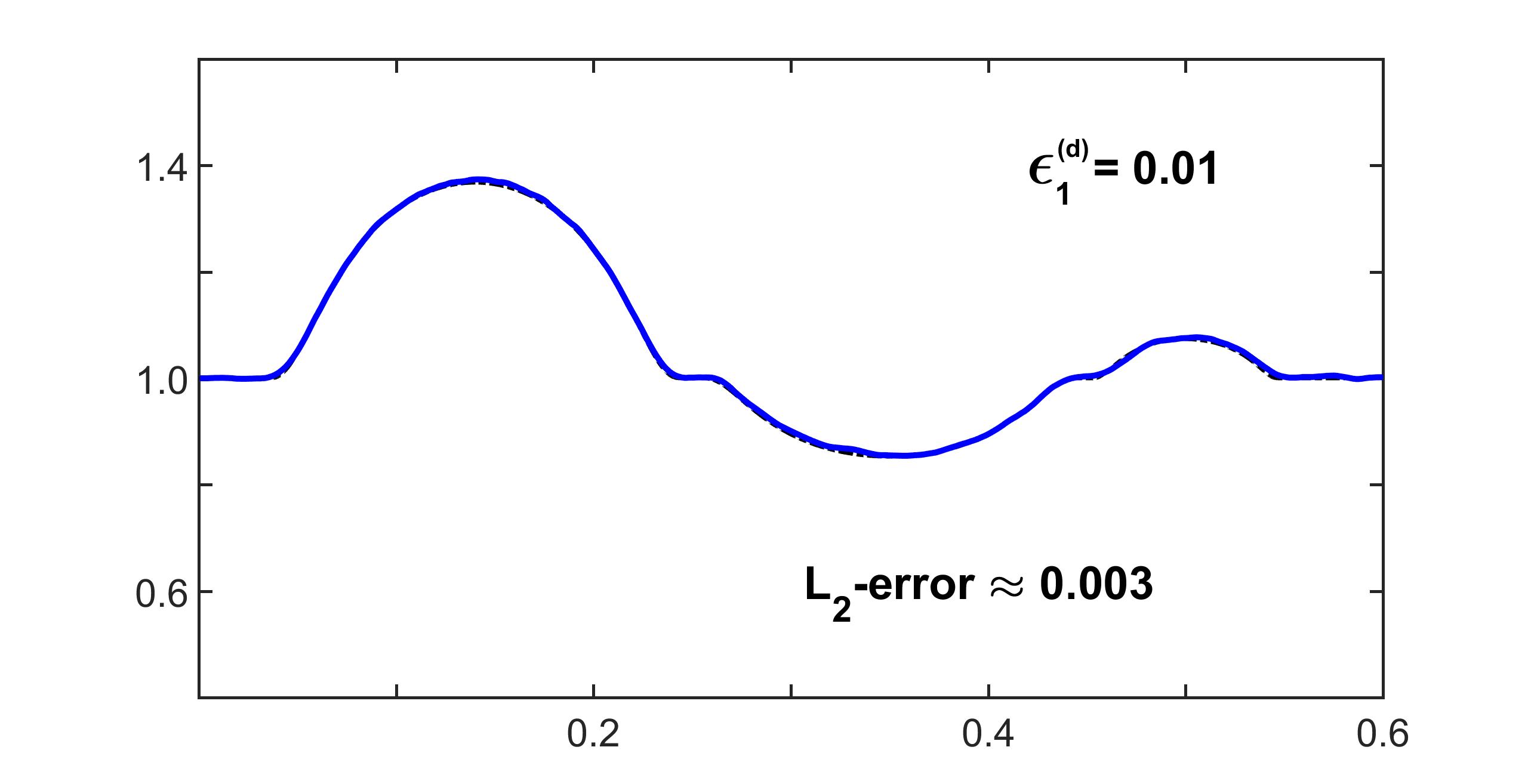}
\caption{
Two reconstructions (the solid blue lines) of a smooth velocity function $c_s$ (the dashed lines).
{\em Top:} Noise level $\epsilon^{(d)}_1= 
 0.1$. {\em Bottom:} Noise level $\epsilon^{(d)}_1= 
 0.01$.
 }
\label{estimatkunkku}
\end{figure}
To study the order of convergence of our reconstruction method, 
we consider 10 noise levels,
\begin{align}
\label{epsihommeli2}
\epsilon^{(d)}_{1}\in\{k\cdot 10^{-2}|k=1,2,3,...,10\},
\end{align} 
and simulate noisy measurements with five different realizations of the random vector
$n$ in (\ref{the_noise}) at each noise level.
The corresponding reconstruction errors 
$\norm{\wtilde c^{N_1}_{\alpha_{1}} -c_s}_{L^2 (M)}$
are summarized in Figure \ref{fourestimates}.
Computations suggest that the order of convergence is 0.40.
This is better than $\frac{1}{54}$ in Theorem (\ref{kaiken_teoria2}).  

\begin{figure}
\centering
\includegraphics[scale=0.18]
{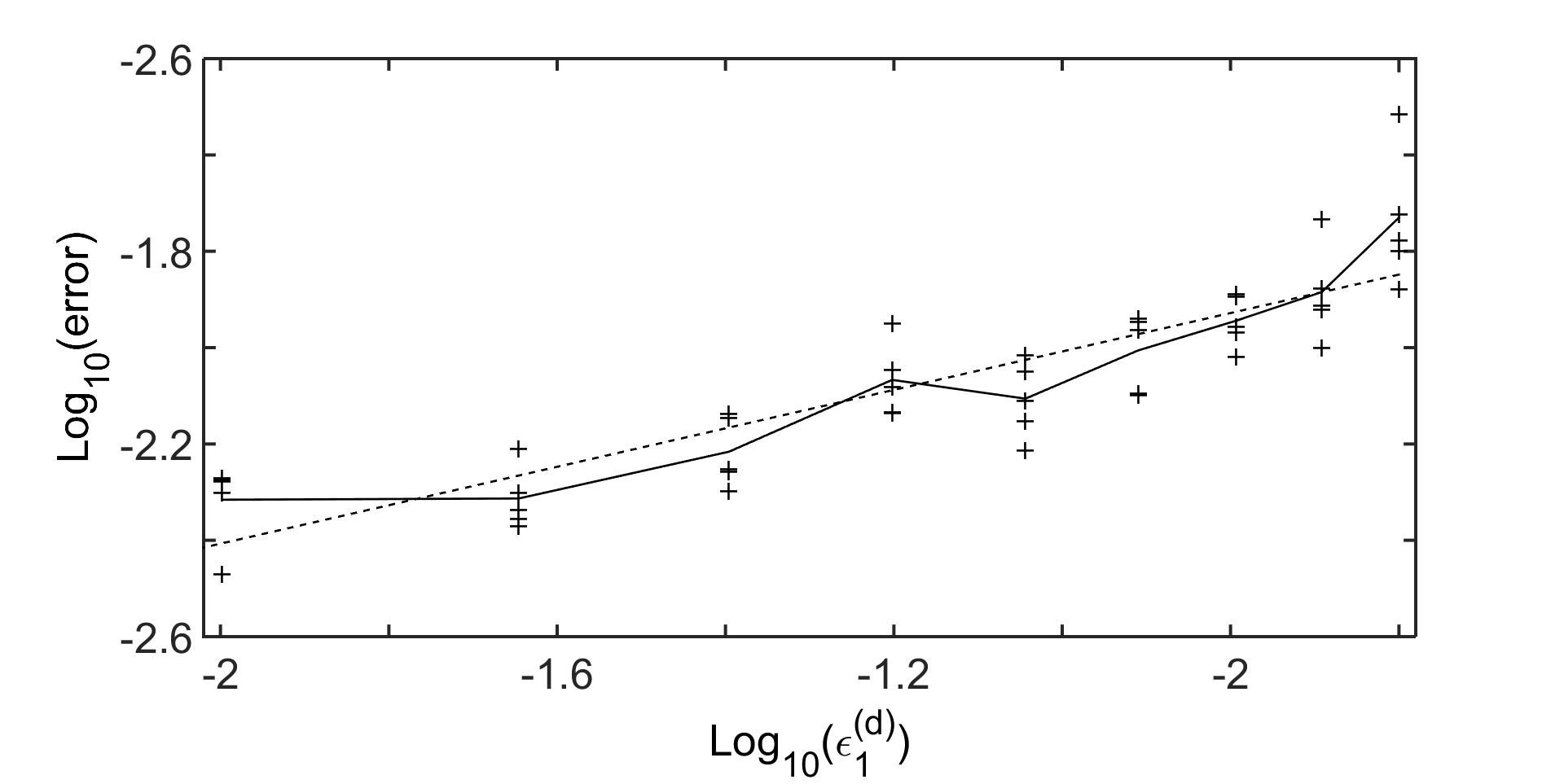}
\caption{
The reconstruction error as a function of the noise level $\epsilon^{(d)}_{1}$ (in log--log axes). We have used here 5 different realizations of the noise. The solid line is the average of these. 
Linear fitting (the dotted line) gives the estimated convergence the order 0.40.
}
\label{fourestimates}
\end{figure}

We also tested the method with a non-smooth velocity function $c_p$(see Figure \ref{estimatikka23}), where reconstructions of two different noise levels are shown. This case is not covered by the above analysis, but the reconstruction method is also robust in this case.

\begin{figure}
\centering
\includegraphics[scale=0.14]
{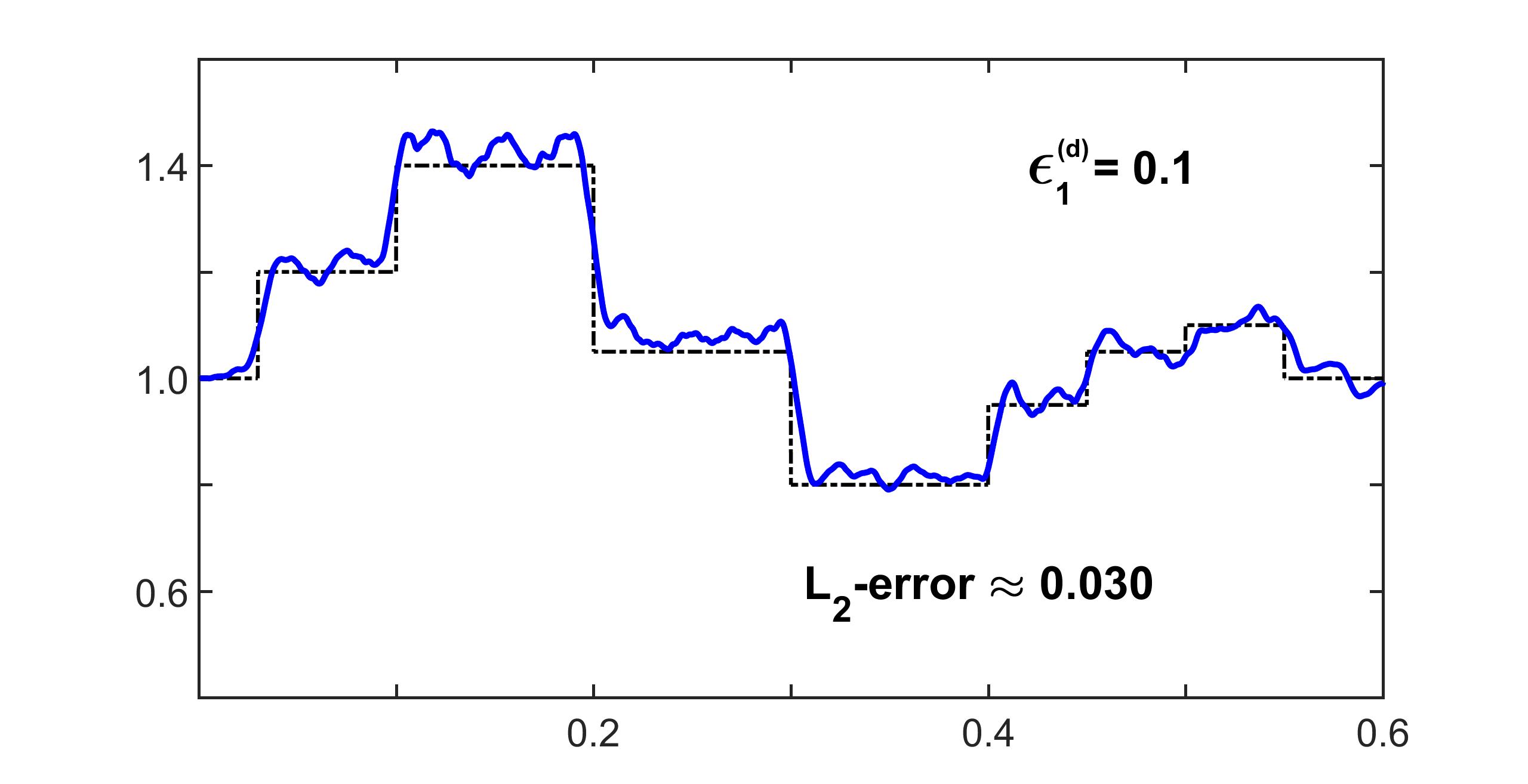}
\includegraphics[scale=0.14]
{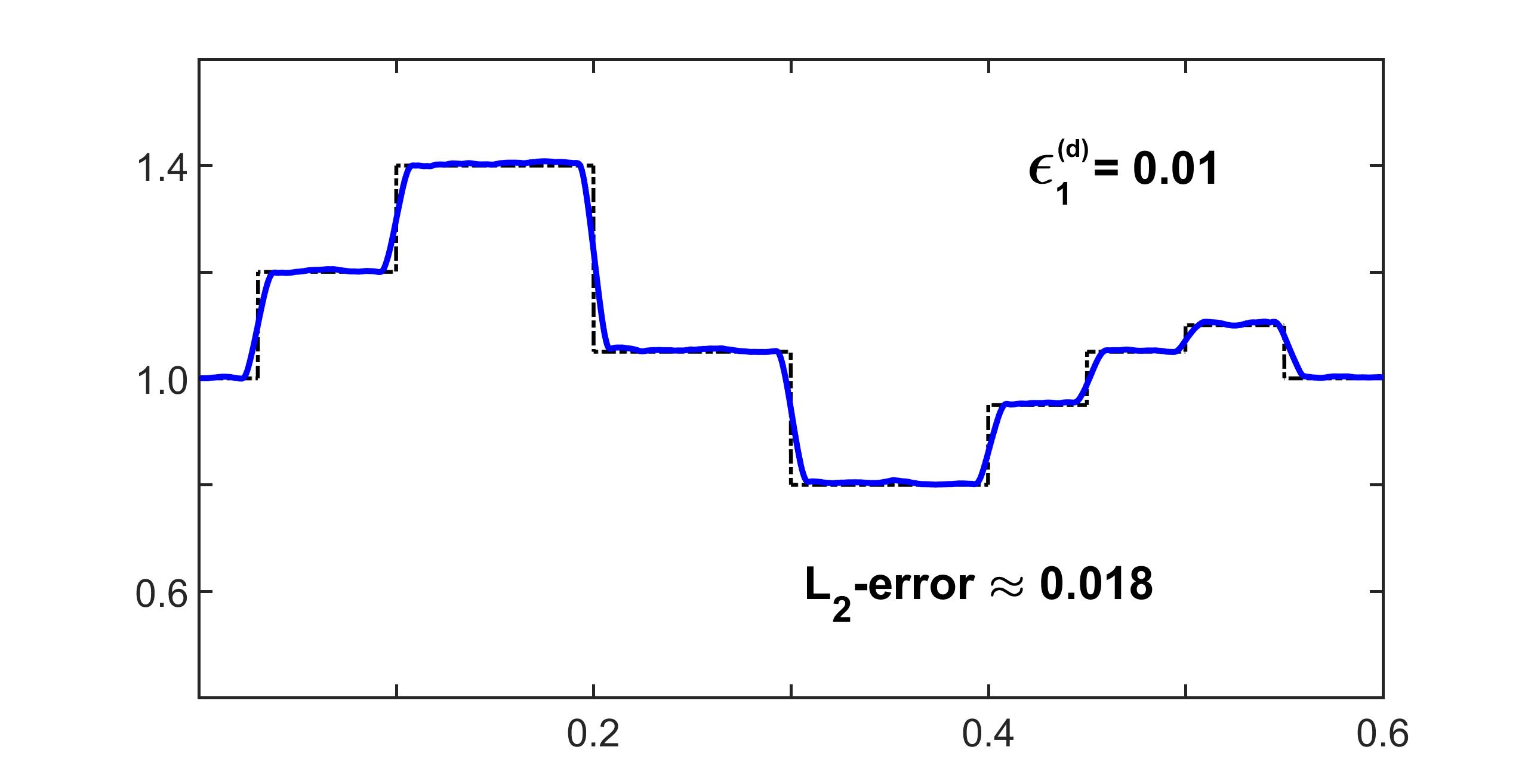}
\caption{
Two reconstructions (the solid blue lines) of a piecewise constant velocity function $c_p$ (the dashed lines).
{\em Top:} Noise level $\epsilon^{(d)}_1= 
 0.1$. {\em Bottom:} Noise level $\epsilon^{(d)}_1= 
 0.01$.
}
\label{estimatikka23}
\end{figure}

\subsection{Reconstruction results based on MDP}
\label{pamelapulkkinen}
Here we use a heuristic version of MDP as a parameter choice
rule for $\alpha=\alpha(\epsilon)$.
Typically MDP is applied to a Tikhonov regularization of the form
\begin{equation}
\label{classical_Tikhonov}
\min_x \norm{F(x) - y}^2 + \alpha \norm{x - x^*}^2,
\end{equation}
where $F$ is a model for the measurements, $y$ is the data, and $x^*$ plays the role of a selection criterion. 
In our case, the model $F$ corresponds to $c \mapsto \Lambda_{N_1}^{c}\phi_{1,N_1}$
and $y = \wtilde\Lambda_{N_1}^{c}\phi_{1,N_1}$ gives the measurement data, but we do not cast the inverse problem as a minimization problem and our regularization method is not of the Tikhonov type.
In particular, our method does not depend on the choice of the auxiliary parameter $x^*$, which can be viewed as an initial guess, and that chooses a local minimum of the non-linear optimization problem (\ref{classical_Tikhonov}).
Due to these differences, the existing results on MDP do not apply to our method, and (\ref{morozov}) below is only a heuristic analogue of the classical MDP.
We refer to \cite{Scherzer1993} for a study of MDP in an abstract context of the form (\ref{classical_Tikhonov}), with non-linear $F$. 

The heuristic principle that we use is as follows. We fix tuning parameters $h > 1$ and small $\delta > 0$ and search for a regularization parameter $\alpha$ in such a way that the following consistency condition holds:
\begin{align}
\label{morozov}
(h-\delta)\epsilon\le\norm{\Lambda_{N_1}^{\wtilde c^{N_1}_{\alpha}}\phi_{1,N_1}-y}_{L^2(0,2T)}
\le(h+\delta)\epsilon.
\end{align} 
Here $\epsilon > 0$ is the noise level, 
$y = \wtilde\Lambda_{N_1}^{c}\phi_{1,N_1}$ is again the measurement data, and $\Lambda_{N_1}^{\wtilde c^{N_1}_{\alpha}}\phi_{1,N_1}$ is the corresponding data computed with the velocity function $\wtilde c^{N_1}_{\alpha}$, given by the reconstruction method. 
Observe that (\ref{morozov}) is a relaxed version of (1.7) in \cite{Scherzer1993}.

We choose $h=1.1$ and $\delta = 0.01$ and use a bisection search to find $\alpha$. Our implementation was unable to find $\alpha$ satisfying the constraint (\ref{morozov}) for noise levels $\epsilon^{(d,k)}>0.02$.
For smaller noise levels, the regularization parameters found using the principle are summarized in Figure \ref{estimatee}.
We see that, with the above choice of tuning parameters, the heuristic MDP always gives a larger regularization parameter than (\ref{the_alpha}). The reconstructions are consistently worse than those produced by the choice (\ref{the_alpha}).

\begin{figure}
\centering
\includegraphics[scale=0.15]
{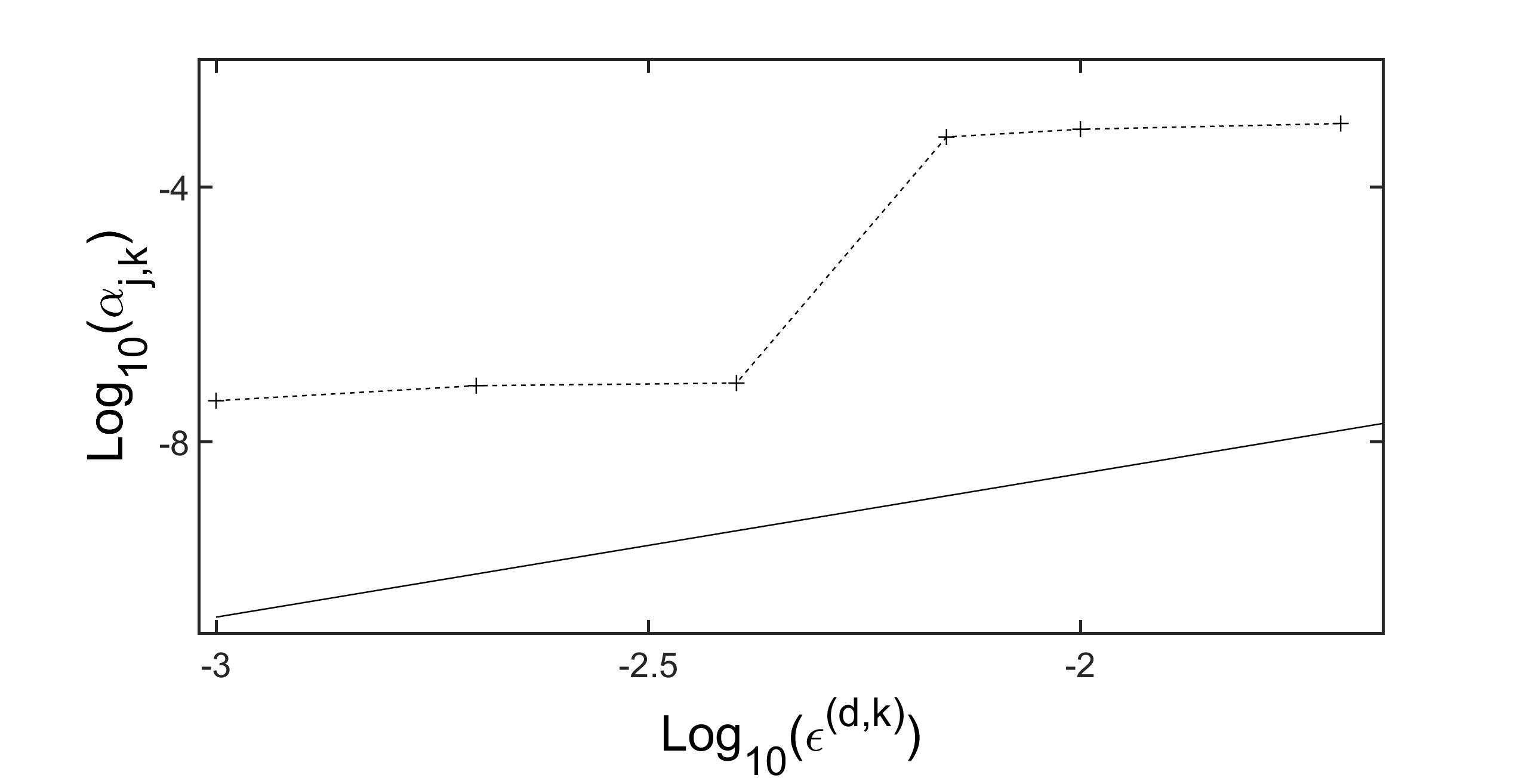}
\caption{
Regularization parameter $\alpha$, given by MDP as a function of the noise level (in log--log axes).
The straight line represents the relation (\ref{the_alpha}).
}
\label{estimatee}
\end{figure}

\noindent{\bf Acknowledgements.} We thank Samuli Siltanen for inspiring discussions on the regularization on inverse problems. 

L.\ Oksanen was partly supported by 
EPSRC, project EP/P01593X/1.
J.\ Korpela and M.\ Lassas were  supported by the Academy of Finland, projects 263235, 273979, 284715, and 312119.

\bibliographystyle{abbrv} 
\bibliography{main}
\end{document}